\documentclass[12pt]{amsart}
\usepackage{amsmath,amsthm,amssymb,todonotes,a4wide}
\usepackage{pdfsync}

\usepackage{esint}
\usepackage[all]{xy}

\usepackage{a4wide}

\theoremstyle{plain}

\newtheorem{thmA}{Theorem}
 \newtheorem{theorem}{Theorem}[section]
 \newtheorem{lem}[theorem]{Lemma}
 \newtheorem{prop}[theorem]{Proposition}
 \newtheorem{proposition}[theorem]{Proposition}

 \newtheorem{cor}[theorem]{Corollary}

 \theoremstyle{definition}
 
 \newtheorem{definition}[theorem]{Definition}
 \newtheorem{defi}[theorem]{Definition}
 
\newtheorem{ex}[theorem]{Example}
\newtheorem{notation}[theorem]{Notation}
 \newtheorem{rem}[theorem]{Remark}

 \numberwithin{equation}{section}
\numberwithin{table}{section}

\newcommand{\GF}{{{\bG^F}}}
\newcommand{\GFeins}{{{\bG^{F_1}}}}

\newcommand{\wbGF}{\w \bG^F}

\newcommand{{\cS}}{{\mathcal S}}
\newcommand{{\cK}}{{\mathcal K}}
\newcommand{{\cN}}{{\mathcal N}}
\newcommand{{\cW}}{{\mathcal W}}

\newcommand{{\calf}}{{\mathrm f}}
\newcommand{{\calg}}{{\mathfrak g}}
\newcommand{{\calK}}{{\mathcal K}}

\newcommand{\w}{\widetilde}
\newcommand{\wG}{{\widetilde G}}\newcommand{\wbG}{{\widetilde \bG}}

\newcommand{\Gu}{{\w G}}

\newcommand{\Nu}{{\w N}}

\newcommand{\la}{\ensuremath{\lambda}}
\newcommand{\La}{\ensuremath{\Lambda}}

\newcommand{\wxi}{\ensuremath{{\w \xi}}}

\newcommand{\Irrl}{\mathrm{Irr}_{\ell'}}
\newcommand{\calI}{\ensuremath{\mathcal I }}

\renewcommand{\labelenumi}{(\alph{enumi})}

\renewcommand{\theenumi}{\thetheorem{}(\alph{enumi})}

\newcommand{\ovC}{\mathbf C}

\newcommand{\ovS}{{\mathbf S}}

\newcommand{\W}{\ensuremath{\mathrm{W}}}

\def\id{\mathrm{id}}

\newcommand{\ZZ}{\ensuremath{\mathbb{Z}}}
\newcommand{\CC}{\ensuremath{\mathbb{C}}}

\newcommand{\Fdual}{{F^*}}
\newcommand{\frakI}{\mathfrak{I}}

\renewcommand{\o}{\overline}

\newcommand{\bB}{{\mathbf B}}

\newcommand{\Zent}{\ensuremath{{\rm{Z}}}}
\newcommand{\Cent}{\ensuremath{{\rm{C}}}}

\newcommand{\NNN}{\ensuremath{{\mathrm{N}}}}

\newcommand{\ep}{\epsilon}
\def\restr#1|#2{\left.#1\right\rceil_{#2}}

\newcommand{\FF}{\ensuremath{\mathbb{F}}}

\newcommand{\E}{\ensuremath{\mathcal{E}}}

\newcommand{\calC}{{\mathcal C}}

\newcommand{\calS}{{\mathcal S}}

\newcommand{\SL}{{\mathrm{SL}}}
\newcommand{\PSL}{{\mathrm{PSL}}}
\newcommand{\PGL}{{\mathrm{PGL}}}
\newcommand{\diag}{\operatorname{diag}}
\newcommand{\GL}{{\mathrm{GL}}}
\newcommand{\GU}{{\mathrm{GU}}}
\newcommand{\SU}{{\mathrm{SU}}}
\newcommand{\PSU}{{\mathrm{PSU}}}

\newcommand{\tA}{\mathsf A}

\newcommand{\Cy}{\mathrm C}

\newcommand{\R}{\mathrm {R}}

\newcommand{\calG}{\ensuremath{\mathcal G}}
\newcommand{\Res}{{\mathrm{Res}}}

\newcommand{\hc}{{h_\calC}}

\newcommand{\bG}{{\bf G}}
\newcommand{\bH }{{\bf H}}
\newcommand{\n}{{\mathbf n}}

\newcommand{\x}{{\mathbf x}}

\newcommand{\al}{{\alpha}}
\newcommand{\si}{{\sigma}}

\newcommand{\Jor}{\operatorname {Jor}}
\newcommand{\Wey}{\operatorname {Wey}}
\newcommand{\sss}{\operatorname {ss}}

\makeatletter
\def \pseudosubsection#1 {\medskip\noindent  {\bf #1}   \smallskip}

\def\Spann<#1>{\Spann@h#1@}
\def\Spann@h#1|#2@{\left\langle\left.#1\vphantom{#2}\hskip.1em\right|\,\relax #2\right\rangle}
\def\Set#1{\Set@h#1@}
\def\Lset#1{\Lset@h#1@}
\def\Set@h#1|#2@{\left\{\left.#1\vphantom{#2}\hskip.1em\,\right|\,\relax #2\right\}}
\def\Bset@h#1in#2|#3@{\Set{{#1\qin#2}|{#3}}}
\def\Lset@h#1@{\left\{#1\right\}}
\def\spann<#1>{\left\langle#1\right\rangle}

\makeatother
\newcommand{\Sym}{{\mathfrak{S}}}

\newcommand{\Aut}{\mathrm{Aut}}
\newcommand{\Fp}{{F_p}}
\newcommand{\Out}{\ensuremath{\mathrm{Out}}}

\newcommand{\bGvFeins}{{\bG^{vF_1}}}
\newcommand{\wbGvFeins}{{\wbG^{vF_1}}}

\newcommand{\wbTvF}{{(\w\bT)^{vF}}}

\newcommand{\wN }{\ensuremath{{\widetilde N}}}

\newcommand{\calM}{\ensuremath{\mathcal M}}
\newcommand{\calN}{\ensuremath{\mathcal N }}

\newcommand{\sgn}{\ensuremath{\mathrm {sgn}}}

\newcommand{\calZ}{\mathcal Z}
\newcommand{\Z}{\operatorname Z}

\newcommand{\bU}{\mathbf U}
\newcommand{\bL}{\mathbf L}
\newcommand{\bC}{\mathbf C}
\newcommand{\bS}{\mathbf S}
\newcommand{\bT}{\mathbf T}
\newcommand{\wbT}{\w \bT}

\newcommand{\calP}{\mathcal P}

\newcommand{\Irr}{\operatorname{Irr}}

\makeatother

\global\long\def\ovG{\mathbf{G}}

\global\long\def\CC{\mathbb{C}}

\global\long\def\calK{\mathcal{K}}

\global\long\def\SL{\operatorname{SL}}
\global\long\def\CC{\mathbb{C}}
\global\long\def\R{\mathrm{R}}
\global\long\def\FF{\mathbb{F}}

\global\long\def\NNN{\mathrm{N}}

\global\long\def\al{\alpha}

\newcommand{\lp}{{\ell '}}

\newcommand{\HH}{{\mathbf{H}}}
\newcommand{\Ind}{{\mathrm{Ind}}}

\newcommand{\Omegau}{\w \Omega}
\newcommand{\Guchinull}{\w G_{\chi_0}}

\newcommand{\enumroman}{\renewcommand{\labelenumi}{(\roman{enumi})} \renewcommand{\theenumi}{\thetheorem(\roman{enumi})}}
\newcommand{\enumalph}{\renewcommand{\labelenumi}{(\alph{enumi})} \renewcommand{\theenumi}{\thetheorem(\alph{enumi})}}

\def\norm#1#2{{\operatorname N}_{#1}(#2)}
\def\cent#1#2{{\operatorname C}_{#1}(#2)}

\renewcommand{\to}{\rightarrow}

\renewcommand{\wbG}{{\w\ovG}}
\newcommand{\wGF}{{\w \ovG^F}}
\newcommand{\wGFeins}{{\w \ovG^{F_1}}}

\newcommand{\bX}{\mathbf X}
\renewcommand{\bU}{\mathbf U}
\newcommand{\bP}{\mathbf P}

\newcommand{\gl}{\mathfrak {gl}}

\newcommand{\Uni}{\operatorname{Uni}}

\newcommand{\wh}{\widehat}

\newcommand{\oFF}{\o \FF}

\newcommand{\und}{\text{ and }}
\newcommand{\forevery}{\text{ for every }}

\newcommand{\oFp}{{\o \FF_p}}
\newcommand{\deq}{\mathrel{\mathop:}=}

\makeindex

\title[Correspondences and inductive McKay condition for type $\tA$]{Equivariant character correspondences and inductive McKay condition for type $\tA$}

\begin{document}


\abstract{As a step to establish the McKay conjecture on character degrees of finite groups, we verify the inductive McKay condition introduced by Isaacs-Malle-Navarro for simple groups of Lie type $\tA_{n-1}$, split or twisted. Key to the proofs is the study of certain characters of $\SL_n(q)$
and $\SU_n(q)$ related to generalized Gelfand-Graev representations. As a by-product we can show that a Jordan decomposition for the characters of the latter groups is equivariant under outer
automorphisms. Many ideas seem applicable to other Lie types.}

\author{Marc Cabanes \and Britta Sp\"ath}
\address{M. Cabanes: Institut de Math\'ematiques de Jussieu, Universit\'e Paris Diderot,
B\^atiment Sophie Germain, 75205 Paris Cedex 13, France.} \email{cabanes@math.jussieu.fr}

\address{B. Sp\"ath: Fachbereich Mathematik, TU {Kaisers}lautern, Postfach 3049, 67653 Kaisers-{lautern}, Germany. }
\email{spaeth@mathematik.uni-kl.de}
\thanks{The second author has been supported by the Deutsche Forschungsgemeinschaft, SPP 1388.}}
\maketitle

\setcounter{tocdepth}{1} 
{\small{\tableofcontents{}}}

\section{Introduction}
\noindent
In the representation theory of finite groups one of the most intriguing conjectures is the McKay conjecture. It is a quite elementary  global/local statement claiming that 
\[|\Irr_\lp (G)|=|\Irr_\lp (\NNN_G(P))|\]
 for any finite group $G$ and prime number $\ell$, where $P$ is a Sylow $\ell$-subgroup of $G$ and where, for any finite group $H$, one denotes by $\Irr_\lp (H)$ the set of irreducible (complex) characters whose degrees are prime to $\ell$.

The reduction theorem by Isaacs, Malle and Navarro in \cite{IsaMaNa} gives hope that this conjecture can be proved by use of the classification of the finite simple groups. Assuming the so-called {\it inductive McKay condition} from \cite[\S 10]{IsaMaNa} for every finite simple group they could show the above statement for any finite group. In several cases the inductive McKay condition has been verified: simple groups not of Lie type, groups of certain Lie types, and groups of all Lie types when $\ell$ is the defining prime \cite{ManonLie, CabSpaeth, Spaeth5}. It is also expected that a verification of this inductive condition for all Lie types gives new insights into the conjecture and paves the way for an approach to its block version.

The inductive McKay condition is essentially two-fold. The first half requires for each finite quasi-simple group $G$ a bijection $\Irr_\lp (G)\to\Irr_\lp (\NNN_G(P))$ that must be equivariant for all automorphisms of $G$ stabilizing $P$. The second requirement is of a cohomological nature and relates to cocycles on subgroups of $\Out(G)$ (see \cite[10.(8)]{IsaMaNa}). Special linear and special unitary groups are therefore seemingly the most difficult with respect to this cohomological criterion. The goal of this paper is to verify the inductive McKay condition for them.

\begin{thmA}\label{thm_PSL_ist_gut} The simple groups $\PSL_n(q)$ and $\PSU_n(q)$ satisfy the inductive McKay condition. \end{thmA}

We also consider the so-called IN-refinement of the inductive McKay condition from \cite[3.1]{Spaeth_AM_red} corresponding to the refinement of the McKay conjecture introduced by Isaacs-Navarro in \cite{IsaacsNavarro}. We check that it holds in this situation, see Theorem~\ref{thm_IN}.

Let us recall the convention of putting $\SU_n(q)=\SL_n(-q)$, $\GU_n(q)=\GL_n(-q)$, thus allowing us to write $\SL_n(\epsilon q)$ and $\GL_n(\epsilon q)$ for $\epsilon =\pm 1$.

Our main statements are proven by applying the criterion from \cite[2.12]{Spaeth5}. The requirements of this criterion can be divided into three main parts: a {\it global} statement about stabilizers and extendibility of characters of $G=\SL_n(\epsilon q)$, a {\it local} statement analogous to the first about characters of a ``local subgroup" playing the r\^ole of $\NNN_G(P)$, and a {\it bijection} between some characters of $\w G =\GL_n(\epsilon q)$, on the one hand, and some characters of local subgroups, on the other hand.

Our proof uses certain characters of $\SL_n(\epsilon q)$ related to Kawanaka's generalized Gelfand-Graev representations of $\GL_n(\epsilon q)$ from \cite{Kaw_exc}. By our constructions, each of those generalized characters is invariant under field and graph automorphisms, and also extends to
a corresponding semi-direct product. From this we deduce a similar property for characters of
$\SL_n(\epsilon q)$, see Theorem \ref{GloStaA}. This might be interesting in its own right. Indeed, a related question is the equivariance of the Jordan decomposition of characters of $\SL_n(\epsilon q)$. In Theorem~\ref{EquJor} we show that our observations imply that Bonnaf\'e's construction of a Jordan decomposition of characters from \cite[5.3]{Bo00} gives one which is equivariant with respect to outer automorphisms.

For the study of the local situation, one considers the normalizers of so-called Sylow $\Phi_d$-tori of $\SL_n(\epsilon q)$ from \cite{BrMa}. Their characters have earlier been parametrized in \cite{Spaeth3, Spaeth2} but we strengthen these results in order to describe in addition the action of automorphisms on the characters and extendibility properties.

As the last ingredient for applying the criterion from \cite[2.12]{Spaeth5}, we establish a character correspondence with additional equivariance properties, using Jordan decomposition of characters, generalized $d$-Harish-Chandra theory and a so-called extension map.

It will be clear to the experts that many arguments can be given easy variants in other types, classical or exceptional. We did not try to give them here since the key ingredient contained in Theorem \ref{Mult1} below has no clear analogue in other types.

This paper is structured in the following way. In Sect.~\ref{sec2} we introduce the general notation around characters, and recall \cite[2.12]{Spaeth5}. Section \ref{sec3} introduces our notation around the considered simple groups and some of their group theoretic properties. In Sect.~\ref{sec4} we prove the ``global" step described above. Then we consider the local situation in Sect.~\ref{sec5} and check that the condition on those characters given in \cite[2.12]{Spaeth5} is satisfied. Finally in Section \ref{sec_Bij_Gu} we construct a bijection, similar to the one of Malle in \cite{Ma06} but with additional equivariance properties, see Theorem \ref{thm6_1}. Section~\ref{sec7} gives the last step of the proof of Theorem \ref{thm_PSL_ist_gut}, while Section \ref{sec8} describes the equivariant Jordan decomposition of characters.

\medskip

\noindent {\bf Acknowledgement:} We thank C\'edric Bonnaf\'e for very useful comments and Jay Taylor for several explanations about generalized Gelfand-Graev characters. The first author thanks the DFG Priority Program SPP 1388 and the ERC Advanced Grant 291512 for the support of a stay in Kaiserslautern.

\section{Notation and known results}  \label{sec2}
\noindent
 In this section we establish the notation around groups and characters that is used throughout this paper. Additionally we recall a criterion for the inductive McKay condition that has been established in \cite{Spaeth5}. Theorem \ref{thm2_2} can be seen as a guideline for the paper. 

For groups and their characters we use the notation introduced in \cite{Isa}. 

\begin{notation}[Characters and group actions]
If a group $A$ acts on a finite set $X$ we denote by $A_{x}$ the stabilizer of $x\in X$ in $A$, analogously we denote by $A_{X'}$ the setwise stabilizer of $X'\subseteq X$. If $A$ acts on a group $G$ by automorphisms, there is a natural action of $A$ on $\Irr(G)$ given by 
\[ {}^{a^{-1}}\chi (g)=\chi^a(g)=\chi(g^{a^{-1}}) \text{ for every } g \in G,\,\, a\in A \und \chi\in\Irr(G).\]
For $P\leq G$ and $\chi\in \Irr(H)$ for some $A_P$-stable subgroup $H\leq G$, we denote by $A_{P,\chi}$ the stabilizer of $\chi$ in $A_P$. 

We denote the restriction of $\chi\in\Irr(G)$ to some subgroup $H\leq G$ by $\Res^G_H(\chi)$, while $\Ind^G_H(\psi)$ denotes the character induced from $\psi\in\Irr(H)$ to $G$. Additionally, for $N\lhd G$ we sometimes identify the characters of $G/N$ with the characters of $G$ whose kernel contains $N$. 

For $N\lhd G$ and $\chi\in \Irr(G)$ we denote by $\Irr(N\mid \chi)$ the set of irreducible constituents of the restricted character $\Res^G_N(\chi)$, and for $\psi \in \Irr(N)$, the set of irreducible constituents of the induced character $\Ind_N^G(\psi)$ is denoted by $\Irr(G\mid \psi)$. For a subset $\calN\subseteq \Irr(N)$ we define
\[ \Irr(G\mid \calN)\deq\bigcup_{\chi\in\calN}\Irr(G\mid \chi).\]
For a prime $\ell$ and an integer $i$ we denote by $i_\ell$ the biggest power of $\ell$ dividing $i$ and let $\Irr_{\ell '}(G)\deq\{\chi\in\Irr (G)\mid \chi (1)_\ell =1 \}$. 
\end{notation}

\begin{theorem}[{\cite[2.12]{Spaeth5}}]\label{thm2_2}
Let $S$ be a finite non-abelian simple group and $\ell$ a prime with $\ell\mid |S|$. Let $G$ be the universal covering group of $S$ and $Q$ a Sylow $\ell$--subgroup of $G$. Assume there exist groups $A$, $\w G\leq A$, $D\leq A$ and $N<G$, such that for them and  $\wN\deq N\NNN_{\wG}(Q)$, the following statements hold: 
\enumroman
\begin{enumerate}
\item \label{thm2_2i}
\begin{itemize}
\item $G\lhd A$ and $A=\wG \rtimes D$,
\item $\wG/G$ is abelian,
\item $\Cent_{\wG\rtimes D}(G)= \Z(\wG)$
and $A/\Z(\wG)\cong\Aut(G)$ by the natural map,
\item $N$ is $\Aut(G)_Q$-stable,
\item $\NNN_G(Q)\leq N\neq G$,
\item every $\chi\in\Irrl(G)$ extends to its stabilizer $\w G_\chi$ ,
\item every 
 $\psi\in \Irrl(N)$ extends to its stabilizer $\wN_\psi$.
\end{itemize} 
\item \label{thm2_2iii}
Let $ \calG\deq \Irr\left (\Gu\mid \Irrl(G)\right )$. For every $\chi\in \calG$ there exists some $\chi_0\in \Irr (G\mid \chi)$ such that 
\begin{itemize}
\item $(\Gu\rtimes D)_{\chi_0}= \Guchinull\rtimes D_{\chi_0}$ and
\item $\chi_0$ extends to $(G \rtimes D)_{\chi_0}$.
\end{itemize} 
\item \label{thm2_2iv}
Let $\calN\deq \Irr\left (\Nu\mid \Irrl(N)\right )$. For every $\psi\in \calN$ there exists some $\psi_0\in \Irr(N\mid \psi)$ such that 
\begin{itemize}
\item $O= (\Gu\cap O) \rtimes (D\cap O)$ for $O\deq (\Gu\rtimes D)_{N,\psi_0}G$ and
\item $\psi_0$ extends to $(G\rtimes D)_{N,\psi_0}$.
\end{itemize}
	\item \label{thm2_2ii} There exists a $(\wG\rtimes D)_Q$-equivariant bijection 
	\[\w \Omega: \calG \longrightarrow \calN\] 
	with 
	\begin{itemize}
		\item $ \Omegau(\calG\cap\Irr(\Gu\mid \nu))= \calN\cap\Irr(\Nu\mid \nu)$ for every $\nu \in \Irr(\Z(\Gu))$,
		\item \label{Omega_u_epsilon_equiv} $\Omegau(\chi\delta)= \Omegau(\chi)\Res_{\w N}^{\w G}( \delta)$ for every $\chi\in \calG$ and every $\delta \in \Irr(\Gu)$ with $G\leq \ker(\delta)$.
	\end{itemize}
\end{enumerate}
\enumalph
\smallskip
Then the inductive McKay condition from \cite[\S 10]{IsaMaNa} holds for $S$ and $\ell$.
\end{theorem} 

\section{Simple groups of type $\tA$}\label{sec3}
\noindent
 We prove Theorem \ref{thm_PSL_ist_gut} by applying Theorem~\ref{thm2_2} in the case where $S$ is a projective special linear or projective special unitary group.

The universal covering group of a simple group $S\in\Lset{ \PSL_n(q), \PSU_n(q)}$ is a group isomorphic to $\SL_n(q)$ or $\SU_n(q)$, apart from a few exceptions, see \cite[6.1.8]{GLS3}. While we consider those exceptions later in Sect.~\ref{sec7}, we study the groups $\SL_n(q)$ and $\SU_n(q)$ and their representations in the main part of the paper. 

\medskip\subsection{Groups and subgroups}\label{not_bG_wbG}
\hfill\break

 \noindent Let $p$ be a prime and $q=p^m$ for some positive integer $m$. Let $\o\FF_p$ be the algebraic closure of the field $\FF_p$ with $p$ elements. Let $n\geq 2$ and $\bG=\SL_n(\oFp )\leq \wbG =\GL_n(\oFp )$. 

We denote by $\w\bB$ and $\w\bT$, respectively the group of upper triangular matrices in $\wbG$, respectively of diagonal matrices. Let $\bU$ be the unipotent radical of $\w \bB$. We identify the Lie algebra of $\wbG$ with the algebra $\mathfrak{gl}_n (\oFp )$ of $n\times n$ matrices with coefficients in $\oFp$. Let $\Sigma$ be the root system of $\wbG$ and $\bG$ with respect to $\wbT$ and $\bT\deq \wbT\cap \bG$, respectively. Let $\Sigma^+$ be the set of positive roots and $\Delta\subseteq \Sigma^+$ a basis determined by $\bU$. We classically identify $\Sigma$ with the set $\{e_i-e_j\mid 1\leq i,j\leq n\ ,\ i\not= j\}$ inside a real vector space with orthonormal basis $(e_i)_{1\leq i\leq n}$ (see for instance \cite[1.8.8]{GLS3}). 
Via this identification we have $\Sigma^+=\{e_i-e_j\mid 1\leq j<i\leq n \}$ and $\Delta =\{e_2-e_1,\ldots ,e_n-e_{n-1}\}$.

For every $\al\in\Sigma$ we fix a morphism $\x_\al: (\oFp ,+)\longrightarrow \bG$ of algebraic groups, such that whenever $\alpha =e_i-e_j$ the matrix $\x_\al (t)-\mathrm{\id}_n$ is an elementary matrix with $t$ at the position $(i,j)$. One denotes $\bX_\al\deq\x_\al (\oFp )$.

\medskip\subsection{Frobenius endomorphisms, automorphisms of $\bG$}\label{not_G_wG} 
\hfill\break

\noindent Let $\Fp: \wbG\longrightarrow\wbG$ be the Frobenius endomorphism given by $\Fp: (a_{i,j})\longmapsto (a_{i,j}^p)$. Note that $\Fp$ defines an $\FF_p$-structure on $\wbG$ and $\bG$ and satisfies $\Fp (\x_\al (t))=\x_\al (t^p)$ for any $t\in\oFp$ and $\al\in\Sigma$. For any positive integer $i$ we denote $\Fp^i$ by $F_{p^i}$.

Let $\gamma_0:\wbG\longrightarrow\wbG$ be the automorphism of $\wbG$ and $\bG$ defined by $ g\mapsto (g^ \perp)^{-1}$, where $g^ \perp$ is the transpose of $g$. Let $v_0\in \wbG$ be the matrix with $(-1)^{k+1}$ at the position $(k,n+1-k)$ ($k=1,\dots ,n$) and 0 elsewhere, see \cite[2.7]{GLS3}. Note that by definition $v_0\in\bG$. Then the automorphism \[\gamma:\wbG\longrightarrow \wbG\text{ with } g\mapsto v_0\gamma_0(g)v_0^{-1}\]
satisfies $\gamma (\x_\al (t))=\x_{\gamma(\al )}(t)$ for any $t\in \oFp$ and $\al\in\Delta\cup (-\Delta)$, where $\gamma$ also denotes the automorphism of $\Sigma$ given by 
\[(e_i-e_j) \longmapsto (e_{n+1-j}-e_{n+1-i}).\]
More generally $\gamma(\x_{\al}(t))=\x_{\gamma(\al)}((-1)^{|\al|+1}t)$ for every $\al\in\Sigma$, where $|e_i-e_j|\deq |i-j|$.

Let $F\in\Lset{F_q,\gamma\circ F_q}$ and $\epsilon\in\Lset{-1,1}$ 
with $F= (\gamma)^{\frac{1-\epsilon} 2}\circ F_q$. Recall the notation $\GL_n(- q)=\GU_n(q)$, $\SL_n(-q)=\SU_n(q)$, so that $\bG^F=\SL_n(\epsilon q)$. Composed automorphisms, such as $\gamma\circ F_q$ are often abbreviated as $\gamma F_q$. When $x\in\wbG$, $xF$ generally stands for $\mathrm{int} (x)\circ F$ where $\mathrm{int}(x)$ is the interior automorphism $g\mapsto xgx^{-1}$.

Since $\wGF$ is finite, $F$ is a Steinberg endomorphism in the sense of Definition 21.3 of \cite{MalleTesterman}. Note that we often apply the theorem of Lang-Steinberg on maps $x\mapsto x^{-1}F(x)$ (see \cite[4.4.17]{Springer} or \cite[21.7]{MalleTesterman}) without further reference.

Let $D=\spann<\gamma,\Fp>\leq \Aut(\wbG^F)$. Then the group $\GF\rtimes D$ is well-defined. For any $F$-stable subgroup $\bS_0\leq \bG$ the normalizer of $\bS_0$ in $\GF\rtimes D$ is well-defined and we denote it by $(\GF\rtimes D)_{\bS_0}$.
Whenever $\GF/\Zent (\GF )$ is non-abelian simple and $n\geq 3$, $D$ injects in $\Aut (\GF )$, and the automorphisms of $\GF$ induced by $\wGF$ generate $\Aut(\GF)$ together with $D$, see Theorem 2.5.1 of \cite{GLS3}. This implies 
\begin{align}\label{eq_AutG} (\wGF\rtimes D)/\Z (\wGF) \cong \Aut(\GF),
\end{align}
in those cases. If $n=2$, the automorphisms of $\GF$ induced by $(\wGF\rtimes D)$ generate $\Aut(\GF )$. 

\medskip\subsection{Sylow $\ell$-subgroups and $\Phi_d$-tori} \label{not_d} 
\hfill\break

\noindent In the following, let $\ell$ be a prime with $\ell\mid |\bG^F|$ and $\ell\neq p$. Let $d$ be the multiplicative order of $q$ in $\ZZ/\ell\ZZ$ for odd $\ell$, or the multiplicative order of $q$ in $\ZZ/4\ZZ$ for $\ell=2$. Denote by $\Phi_d(X)\in\ZZ [X]$ the $d$-th cyclotomic polynomial. 

For a pair $(\bH,F)$ consisting of an algebraic group and a Steinberg endomorphism, Brou\'e and Malle introduced in \cite{BrMa} the notion of $\Phi_d$-tori and Sylow $\Phi_d$-tori of $(\bH,F)$, see also \cite[\S 13]{CabEng}. By \cite[3.4]{BrMa} there exists an $F$-stable Sylow $\Phi_d$-torus $\bS_0$ of $(\bG,F)$.

In Sect.~\ref{sec7} we see that in general the group $\NNN_\bG(\bS_0)^F$ satisfies the assumptions made for $N$ in Theorem \ref{thm2_2}. Accordingly we consider in Sect.~\ref{sec5} its characters.

\section{Stabilizers and extensions of characters of $\GF$} \label{sec4}
\noindent
 In this section we consider the stabilizers of characters of $\GF$ in $\wGF\rtimes D$, and essentially verify the assumption made in \ref{thm2_2iii} for $G=\GF$ and $\wG=\wGF$. Note that we are proving here a more general statement since we consider all characters of $\GF$, and make no assumption on their degree.

 \begin{theorem}\label{GloStaA}
 For any $\w\chi\in\Irr (\wGF)$, there is a $\chi_0\in\Irr (\GF \mid\w\chi)$ such that $(\wGF \rtimes D)_{\chi_0}=(\wGF )_{\chi_0}\rtimes D_{\chi_0}$ and $\chi_0$ extends to a character of $\GF \rtimes D_{\chi_0}$.\end{theorem}
 
The proof will be completed in Sect.~\ref{subsec4.3}. We use the character afforded by the generalized Gelfand-Graev representation associated with unipotent classes of $\bG$. This character is constructed using the root subgroups $\bX_\al \leq \bU$ for $\al\in\Sigma^+$. 
 
 \medskip\subsection{Unipotent classes and automorphisms of $\bG$}\label{subsec4.1}
 \hfill\break

\noindent For $\bH$ an algebraic group, we denote by $\Uni (\bH)$ the set of $\bH$-conjugacy classes of unipotent elements of $\bH$.

For $\w\bG =\GL_n(\oFp)$ and $\bG =\SL_n(\oFp)$ the sets $\Uni (\w\bG) = \Uni ( \bG)$ coincide since $\w\bG =\bG\Zent (\w\bG)$. Every class in $\Uni (\w\bG )$ corresponds to a partition of $n$. An easy computation with the Jordan normal form shows that $D$ acts trivially on $\Uni (\w\bG )$.  

Note that for every $g\in \wbG$ the group $\Cent_\wbG(g)$ is connected according to \cite[III.3.22]{SpringerSteinberg}. Hence the theorem of Lang-Steinberg proves that $\calC\longmapsto \calC^{F}$ defines a bijection between $\Uni(\wbG)$ and the set $\Uni (\w\bG ^F)$ of unipotent classes of $\w\bG^F$. (Here we denote by $\calC^F$ the $F$-fixed points in $\calC$.)

In the following let $\calC\in \Uni (\bG )$. The {\em weighted Dynkin diagram} associated with $\calC$ from \cite[2.1.1]{Kaw_exc} can be seen as a map from $\Delta$ to $\{ 0,1,2\}$. 
We also denote by $\hc:\ZZ\Delta\to\ZZ$ the linear map that extends it. For $i\geq 0$ let $\bU_{i,\calC}\deq \Spann<\bX_\al| \al\in\Sigma^+\ ,\ \hc(\al)\geq i>\leq \bU$. Then $\bP_{1,\calC}\deq \NNN_\bG (\bU_{1,\calC})$ is the parabolic subgroup containing $\bT\bU$ associated with $\hc^{-1} (0)\cap\Delta$. 
Chevalley's commutator formula (see \cite[11.8]{MalleTesterman}) implies that $\bU_{i,\calC}\lhd \bP_{1,\calC}$ and that $\bU_{i,\calC}/\bU_{i+1,\calC}$ is abelian and isomorphic to a direct product of $|\hc^{-1}(i) |$ copies of the additive group $\oFp$. 

According to \cite[2.1.1]{Kaw_exc} the set $\calC\cap \bU_{2,\calC}$ is dense in $\bU_{2,\calC}$, forms a single $\bP_{1,\calC}$-conjugacy class and
\[ \Cent_\bG(u)\leq \bP_{1,\calC} \text { for every } u\in \calC\cap \bU_{2,\calC},\] see also Section 2 of \cite{Geck04}. 
According to \cite[(2.3.1)]{Shoji}, $h_\calC$ is symmetric under the graph automorphism of $\Delta$, therefore $\bP_{1,\calC}$ and each $\bU_{i,\calC}$ are $D$-stable. A given class in $\Uni (\wGF)$ generally splits into several $\GF$-classes. For the proof of Theorem \ref{GloStaA} we have to control this splitting with regard to the action of $\gamma$. 
\begin{prop}\label{gammaAction} Let $\calC\in \Uni (\w\bG) = \Uni ( \bG)$. 
\begin{enumerate}
\item \label{gammaActiona} For some $\Fp$-fixed $u\in \calC\cap\bU_{2,\calC}$ the elements $\gamma (u)$ and $u$ are $\bP_{1,\calC}^{\Fp}$-conjugate.
\item \label{gammaActionb} For some $\gamma F_q$-fixed $u\in \calC\cap\bU_{2,\calC}$ the elements $\Fp (u)$ and $u$ are $\bP_{1,\calC}^{  \gamma F_q}$-conjugate. 
\item For some $u\in \calC\cap \bG^{\Fp}$ the elements $u$ and $\gamma (u)$ are $\bG^{\Fp}$-conjugate.
\end{enumerate}
\end{prop}
\begin{proof}First we prove part (c).  

Let $\underline n\deq (n_1,\dots , n_t)$ be a composition of $n$ associated with $\calC$, i.e. $\underline n$ is a family of positive integers summing up to $n$. Let $j$ and $j'$, respectively be the upper triangular matrices in Jordan normal form corresponding to $\underline{n}$ and $\underline n'\deq (n_t,n_{t-1},\ldots , n_1)$, respectively. 

In a first step we prove that $j$ and $j'$ are $\bG^\Fp$-conjugate. Let $v_0'$ be the permutation matrix corresponding to the ``maximal length" permutation $(1,n)(2,n-1)\dots$ . Then $v_0'jv_0'{}^{-1}$ is the block diagonal matrix associated with the composition $(n_t,\dots ,n_1)$ and {\it lower} unitriangular Jordan blocks of sizes $n_t, \dots ,n_1$. 
Let $k'$ be the matrix of block diagonal form for that composition, whose blocks are permutation matrices for the ``maximal length" permutations. Then 
$kjk^{-1}= j'$ for $k=k'v_0'$. 
This proves that $j$ and $j'$ are $\bG^\Fp$-conjugate if $k\in\bG^\Fp$.

Assume $k\notin\bG^\Fp$. If at least one $n_i$ is odd, let $d$ be the diagonal matrix equal to $-\id_{n_i}$ on the corresponding part of the diagonal and $\id_{n_{i'}}$ elsewhere with regard to the composition $(n_1,\dots ,n_t)$. Then $djd^{-1}=j$ and ${}^k j={}^{kd}j=j'$ with $kd\in\bG^\Fp$.

If all $n_i$'s are even, $n_i=2n'_i$ and $n=2n'$, then the determinant of $v_0'$ is $(-1)^{n'}$ and similarly the one of $k'$ is $\prod_i(-1)^{n'_i}$. This proves that $k\in\bG$ since $n'_1+\cdots +n'_t=n'$. So in all cases $j$ is $\bG^{F_p}$-conjugate to $j'$.

In the next step we prove that $j'$ and $\gamma(j)$ are $\bU^\Fp$- and hence $\bG^\Fp$-conjugate. By definition $\gamma$ satisfies 
\[\gamma (j) = v_0 (j^\perp{})^{-1}v_0.\]
Then $\gamma(j)$ is a block diagonal matrix for the composition $(n_t,n_{t-1},\dots , n_1)$. Each of its blocks is a unipotent upper triangular matrix with $1$'s everywhere on and above the diagonal. Since $j'$ has similar shape with Jordan blocks, we may reduce our proof to the case where $t=1$ and $n_1=n$. But then direct computations prove that $j'$ and $\gamma(j)$ are $\bU^\Fp$-conjugates. (It is also a consequence of properties of regular unipotent elements, see \cite[1.4.6.(iii)]{Kaw1982} or \cite[14.26]{DigneMichel}.)

By the two previous steps we may conclude that $j$ and $\gamma (j)$ are $\bG^{\Fp}$-conjugate. This proves part (c).

Let now $x\in \calC\cap \bU_{2,\calC}$. The quotient $\{y\in \bG \mid y^{-1}\Fp(y) \in\Cent_\bG (x)\}/\bG ^{\Fp}$ corresponds to the $\bG^\Fp$-conjugacy classes contained in $\calC^{\Fp}$, where $y\bG^{\Fp}$ is sent to the $\bG^\Fp$-conjugacy class containing $^{y}x$. Since $x\in \calC\cap\bU_{2,\calC}$ we have $\Cent_\bG (x)\leq \bP_{1,\calC}$. The theorem of Lang-Steinberg proves that
$\{y\in \bG \mid y^{-1}\Fp(y) \in\Cent_\bG (x)\}/\bG ^{\Fp}$ and 
$\{y\in \bP_{1,\calC} \mid y^{-1}\Fp(y) \in\Cent_\bG (x)\}/\bP_{1,\calC} ^{\Fp}$ are isomorphic. Hence the $\bG^\Fp$-conjugacy classes contained in $\calC^\Fp$ are $\bP_{1,\calC}$-conjugate and each class has a non-trivial intersection with $\bU_{2,\calC}$. 

Hence according to (c) there exists some $u\in \calC\cap\bU_{2,\calC}^\Fp$ such that $u$ and $\gamma(u)$ are $\bG^\Fp$-conjugate, i.e. $\gamma(u)=u^g$ for some $g\in\bG^\Fp$. Note that $\gamma (u)\in \calC \cap \bU_{2,\calC}^\Fp$. Since $\calC \cap \bU_{2,\calC}^\Fp$ is a single $\bP_{1,\calC}$-orbit, we have $\gamma (u)=u^h$ for some $h\in\bP_{1,\calC}$. Hence $g$ and $h$ satisfy $gh^{-1}\in\Cent_\bG (u)\leq \bP_{1,\calC}$. This implies $g\in \bP_{1,\calC}^\Fp$ and hence $\gamma (u)$ and $u$ are $\bP_{1,\calC}^{\Fp}$-conjugate. This is (a).

It remains to prove (b). By (c) there is some $v\in \calC^{\Fp}$ and $x\in \bG^{\Fp}$ with $\gamma (v)=v^x$. Then $(\gamma \circ F_q)(v)=v^{F_q(x)}=v^x$. By the theorem of Lang-Steinberg, there exists some $y\in\bG$ with $x=y^{-1}\,\, (\gamma \circ F_q)(y)$. Then $u\deq  {}^yv$ satisfies $(\gamma \circ F_q) (u)=u$. Moreover $\Fp(u)={}^{\Fp(y)}v={}^{\Fp(y)y^{-1}}u$. The fact that $x=y^{-1} (\gamma \circ F_q) (y)$ is $\Fp$-fixed implies ${\Fp(y)y^{-1}}\in \bG^{\gamma F_q}$. Hence we have some $u\in \calC^{\gamma F_q}$ such that $u$ and $\Fp (u)$ are $\bG^{\gamma F_q}$-conjugate.
 
The latter is a version of (c) with $(\gamma ,F_p)$ replaced by $(F_p,\gamma F_p)$. From this we can argue as above, using the theorem of Lang-Steinberg and the structure of $\bP_{1,\calC}$ in order to deduce statement (b).
\end{proof}

\medskip\subsection{Generalized Gelfand-Graev representations}\label{subsec4.2}
\hfill\break

\noindent The main tool in this part is the generalized Gelfand-Graev representation and the character it affords. In order to deduce various properties of the associated character we recall its construction here in detail. 

 Kawanaka introduced in \cite[\S 1.3]{Kaw_Ennola} and \cite[\S 3.1]{Kaw_exc} a character $\psi_u\in\Irr (\bU_{2,\calC}^F)$ associated with $u\in \calC\cap \bU_{2,\calC}^F$. Geck gave in Section 2 of \cite{Geck04} an alternative approach for its construction that we use in the following. 
 
Let $\Sigma_{2,\calC}\deq \Sigma^+\cap \hc^{-1} (2)$. For every $g\in \bU_{2,\calC}$ we denote by $g_\al \in\oFp$ ($\al\in\Sigma_{2,\calC}$) the scalars such that \[g\in\prod_{\al\in\Sigma_{2,\calC}}\x_\al (g_\al)\bU_{3,\calC}.\]
Since $\bU_{2,\calC}/\bU_{3,\calC}$ is abelian, the elements $g_\al$ are well-defined and independent of the order of the factors. Analogously let $u_\al \in\oFp$ ($\al\in\Sigma_{2,\calC}$) such that $u\in\prod_{\al\in\Sigma_{2,\calC}}\x_\al (u_\al)\bU_{3,\calC}$. 

Let $\theta_0: (\FF_{p^{2m}} ,+)\longrightarrow (\CC^\times ,\cdot )$ be an $\Fp$-invariant non-trivial linear character, i.e. $\theta_0$ satisfies $\theta_0 (\zeta )=\theta_0 (\zeta^p)$ for every $\zeta\in\FF_{p^{2m}}$. Then one defines $\psi_u:\bU_{2,\calC}^F\longrightarrow \CC^\times$ by 
\[\psi_u (g)=\prod_{\al\in\Sigma_{2,\calC}}\theta_0(c_\al g_\al u_\al),\] 
a morphism, where $c_\al=-1$. 

This formula uses the fact that the underlying algebraic group is $\wbG=\GL_n(\oFp)$. In the general situation considered in \cite{Kaw_exc} the underlying group is a reductive group $\mathbf H$ with connected center and the constant $c_\al$ is a non-zero scalar in $\oFp$. This scalar is determined by a ``Killing-like" form $\kappa$, that is symmetric, non-degenerate and $\mathbf H$-invariant, and an ``opposition automorphism" on the Lie algebra of $\mathbf H$ from \cite[\S 1.2]{Kaw_Ennola}. In our situation the associated Lie algebra of $\wbG$ is $\gl_n(\oFp)$. The transposition map $N\longmapsto N^{\perp}$ is an opposition automorphism on $\gl_n(\oFp)$. The map $\kappa:\gl_n(\oFp) \times \gl_n(\oFp) \longrightarrow \oFp$ that sends any two matrices $M,N\in \gl_n(\oFp)$ to the trace of $MN$ has the required properties, see \cite[I.5.3]{SpringerSteinberg}. This is how one checks that $c_\alpha =-1$ for any ${\al\in\Sigma_{2,\calC}}$ in our case.

According to \cite[1.3.6.(ii)]{Kaw_Ennola} and \cite[3.1.12]{Kaw_exc} there exists a character $\psi '_u\in\Irr (\bU_{1,\calC}^F)$ with $\Ind_{\bU_{2,\calC}^F}^{\bU_{1,\calC}^F}\,(\psi_u) = |\bU^F_{1,\calC}:\bU_{2,\calC}^F|^{{\frac 1 2}}\,\,\psi '_u$. The character of the associated {\em generalized Gelfand-Graev representation} $\w\Gamma_\calC\in\ZZ_{\geq 0}\Irr (\w\bG^F)$ is defined as 
\[ \w\Gamma_\calC: = \Ind_{\bU_{1,\calC}^F}^{\w\bG^F}(\psi '_u) .\]
Clearly $\w\Gamma_\calC$ coincides with $|\bU^F_{1,\calC}:\bU_{2,\calC}^F|^{-{1\over 2}}\,\,\,\Ind_{\bU_{2,\calC}^F}^{\w\bG^F}\,(\psi_u)$.
By definition $\bU_{1,\calC}^F$ and $\bU_{2,\calC}^F$ are $\bP_{1,\calC}^F \rtimes D$-stable.
\begin{lem}\label{ActionGGGR} For $u\in \calC\cap \bU_{2,\calC}^F$ and $\sigma\in \bP_{1,\calC}^F\rtimes D$ the following holds 
\[\psi _u = (\psi _{\sigma(u)})^\sigma\text{ and \   } \psi '_u = (\psi '_{\sigma(u)})^\sigma.\]
\end{lem}
\begin{proof} Since $\psi '_u= |\bU^F_{1,\calC}:\bU_{2,\calC}^F|^{-{1\over 2}}\,\Ind_{\bU_{2,\calC}^F}^{\bU_{1,\calC}^F}(\psi_u )$, it clearly suffices to check $\psi_{\si (u)}(\si (g))=\psi_u(g)$ for any $g\in\bU_{2,\calC}^F$.

For $\si\in\bP_{1,\calC}^F$ this is known from \cite[2.2]{Geck04}. If $\si =\Fp$, we know that $\Fp(\x_\al(t))= \x_\al(t^p)$ for every $\al\in\Sigma$ and $t\in\oFp$. By its definition $\theta_0$ satisfies $\theta_0(\zeta)=\theta_0(\zeta^p)$ for every $\zeta\in\oFp$. Then by the definition of $\psi_u$ and $\psi_{\Fp(u)}$ we obtain 
\[\psi_{\Fp (u)}(\Fp (g))=\psi_u(g) \text{ for every }g\in\bU_{2,\calC}^F,\]
since $c_\al=-1$ for every $\al\in\Sigma_{2,\calC}$. 

It remains to consider the case where $\si =\gamma$. Recall that $\gamma$ induces a map on $\Sigma$, also denoted by $\gamma$, and satisfying 
\[\gamma (\x_\al (t))=\x_{\gamma(\al)}((-1)^{| \al|+1} t) \text{ for every }\al \in\Sigma^+ \und t\in\oFp,\]
where $|\al|\deq |i-j|$ for $\al=e_i-e_j$, see Sect.~\ref{not_G_wG}. Since $c_\al =-1=c_{\gamma(\al)}$ for every $\al\in\Sigma_2$ the definition of $\psi_u$ implies $\psi_u=(\psi_{\gamma(u)})^\gamma$.
\end{proof}
In order to deduce properties of characters of $\GL_n(q)$ and $\GU_n(q)$ from the character of the generalized Gelfand-Graev representation we use the following multiplicity one statement. 

\begin{theorem}[{\cite[3.2.18.(iii), 3.2.24.(i)]{Kaw_Ennola}}]\label{Mult1} 
Let $\wbG\deq \GL_n(\oFp)$ and $F: \wbG \longrightarrow \wbG$ the Steinberg endomorphism from Sect.~\ref{not_G_wG}. Then for any $\chi\in \Irr (\wbG^F )$ there exists $\calC\in\Uni (\wbG)$ such that $ \chi$ is a constituent of $\w\Gamma_\calC$ with multiplicity one.
\end{theorem}
In \cite[5.3.2]{Bo00} and \cite[2.10]{Shoji} one can find additional details on the proof. Note that there is no restrictive hypothesis on $p$ and $q$. Furthermore note that there $\theta_0$ is chosen to be any non-trivial character of $(\FF_p^{2m},+)$, in particular our choice of $\theta_0$ is allowed. In \cite{Lusztig92} Lusztig related all this with the so-called {\em unipotent support} of irreducible characters.

\medskip\subsection{Extensions of characters of $\GF$}\label{subsec4.3}
\hfill\break

\noindent For the proof of Theorem~\ref{GloStaA} we use the following technical statement.
 
\begin{prop}\label{Imply*} Let $Y\lhd X$ and $E$ a finite group acting on $ X$ stabilizing $Y$. Let $\rho\in\ZZ_{\geq 0}\Irr(Y)$ be an $E$-stable character and $\chi\in \Irr(X)$ a constituent of $\Ind^{X}_{Y}(\rho)$ with multiplicity one.
Then
\enumalph\begin{enumerate}
\item there is a unique irreducible component $\chi_0$ of $\Res^{X}_{Y}(\chi)$ with multiplicity one. This $\chi_0$ satisfies $( X\rtimes E)_{\chi_0}= X_{\chi_0}\rtimes E_{\chi_0}$.
\item If $\rho $ extends to $Y\rtimes E$, then $\chi_0$ extends to $Y\rtimes E_{\chi_0}$.
\item \label{prop4_7c} If $E$ is cyclic, then $\rho$ extends to $Y\rtimes E$. 
 \end{enumerate}
\end{prop}
\begin{proof} Since $\chi$ has multiplicity one in $\Ind_Y^X(\rho) $, Frobenius reciprocity proves that $\rho$ and $\Res^X_Y(\chi)$ have a unique common irreducible constituent $\chi_0$, and this has multiplicity one. Let $x\in X$ and $e\in E$ with $\chi_0^x=\chi_0^e$. Since $\rho $ is $E$-fixed, $\chi_0^e$ is also an irreducible constituent of $\rho$. Clearly $\chi_0^x$ is an irreducible constituent of $\Res^{ X}_{Y} (\chi) $. Since $\chi_0^x=\chi_0^e$, this implies that $\chi_0^x=\chi_0$, $x\in X_{\chi_0}$ and $e\in E_{\chi_0}$. So we get (a).

By the assumption in (b), $\rho$ is the restriction of some $\w\rho\in\ZZ_{\geq 0}\Irr(Y{\rtimes} E)$. Note that there exists some irreducible constituent $\psi$ of $ \Res^{{Y} {\rtimes} E}_{{Y} {\rtimes} E_{\chi_0}}(\w\rho)$
such that $\chi_0$ is a constituent of $\Res^{{Y}{\rtimes} E_{\chi_0}}_Y (\psi)$ with multiplicity one. By Clifford theory $\Res^{Y  {\rtimes} E_{\chi_0}}_Y (\psi)$ is a multiple of $\chi_0$, hence 
$\Res^{{Y}  {\rtimes} E_{\chi_0}}_Y (\psi)=\chi_0$.

Let $\calS$ be a set of representatives of the $E$-orbits in $\Irr(Y)$.
For the proof of (c) we choose for every constituent $\psi\in\calS$ of $\rho$ an extension $\w\psi\in\Irr(Y  {\rtimes} E_\psi)$. Such an extension exists by \cite[11.22]{Isa}. For $\psi\in\Irr(Y)$ let $a_\psi$ be the multiplicity of $\psi$ in $\rho$. Since $\rho$ is $E$-invariant $a_\psi=a_{\psi^e}$ for every $e\in E$. One then checks easily that 
$\sum_{\psi\in\calS} a_\psi \, \Ind_{Y{\rtimes}E_\psi}^{Y {\rtimes} E} (\w\psi)$ is an extension of $\rho$.
 \end{proof}
Later we apply Proposition \ref{Imply*} with $X\deq \w\bG^F$ and $Y\deq \bG^F$. For this purpose we consider a generalized Gelfand-Graev representation whose character satisfies the assumptions on $\rho$ from this statement. 

\begin{theorem}\label{GammaSta} 
Let $\calC\in\Uni(\wbG)$. Then let us choose $u\in \calC$ as in Proposition~\ref{gammaActiona} if $F=F_q$, or as in Proposition~\ref{gammaActionb} if $F=\gamma F_q$, respectively. Then 
$\Gamma_u\deq  \Ind_{\bU_{1,\calC}^F}^{\bG^F}(\psi '_u)$ is $D$-invariant and extends to $\GF\rtimes D$.
\end{theorem}
Note that $\Gamma_u$ depends on the actual choice of $u$ within $\calC^F$.

\begin{proof} Since $u\in \calC\cap\bU_{2,\calC}^F$ the characters $\psi_u$ and  $\psi_u'$ are well-defined.

First we prove that $\Gamma_u$ is $D$-invariant if $F=\gamma \circ F_q $. Thanks to Proposition~\ref{gammaActionb} we have $\Fp(u)={}^tu$ for some $t\in\bP_{1,\calC}^F$ and by Lemma~\ref{ActionGGGR} this implies 
\begin{align*}
\Ind_{\bU_{1,\calC}^F}^{\GF}(\psi '_u)
&=\left( \Ind_{\bU_{1,\calC}^F}^{\GF}(\psi '_{\Fp(u)})\right)^{\Fp}=\left( \Ind_{\bU_{1,\calC}^F}^{\GF}(\psi '_{^tu}) \right)^{\Fp}\\
&=\left( \left( \Ind_{\bU_{1,\calC}^F}^{\GF}(\psi '_u)\right )^{t^{-1}} \right)^{\Fp}=\left( \Ind_{\bU_{1,\calC}^F}^{\GF}(\psi '_u) \right)^{\Fp}.
\end{align*} 
Hence $\Gamma_u$ is $\Fp$-invariant in this case.

If $F= F_q$, we have $\Fp (u)=u$ and $\gamma(u)={}^yu$ for some $y\in\bP_{1,\calC}^F$ by Proposition~\ref{gammaActiona}. Hence considerations as above prove the equality $(\Gamma_u)^{\gamma}=(\Gamma_u)^{\Fp}=\Gamma_u$. 

We verify that $\Gamma_u$ extends to $\GF\rtimes D$. 
If $F=\gamma \Fp^m$, then $D$ is cyclic and Proposition \ref{prop4_7c} applies to $\Gamma_u$.

Assume that $F=F_q$. By the choice of $u$ from Proposition~\ref{gammaActiona} and Lemma~\ref{ActionGGGR}, $\psi_u$ is $\Fp$-invariant and there exists some $y \in \bP_{1,\calC}^F$ such that $\psi_u$ is $\mathrm{int} (y)\circ\gamma$-invariant. Then $\psi_u$ has an extension $\w\psi_u\in\Irr(\bU_{2,\calC}^F\rtimes\spann<\Fp>)$ where $\bU_{2,\calC}^F\rtimes\spann<\Fp>$ is seen as a subgroup of $\GF\rtimes D$. Since $\psi_u$ is linear, $\w \psi_u(\Fp)\neq 0$. The induced character $\Ind_{\bU_{2,\calC}^F \rtimes \spann<F_p>}^{\bU_{1,\calC}^F \rtimes \spann<F_p>}( \w\psi_u)$ is an extension of $\Ind_{\bU_{2,\calC}^F}^{\bU_{1,\calC}^F}\,( \psi_u)$ and therefore its constituents are extensions of $\psi_u'$ to $ \bU_{1,\calC}^F \rtimes \spann<F_p>$. By definition the induced character $\Ind_{\bU_{2,\calC}^F \rtimes \spann<\Fp>}^{\bU_{1,\calC}^F \rtimes \spann<\Fp>}( \w\psi_u)$ satisfies 
\[\left (\Ind_{\bU_{2,\calC}^F \rtimes \spann<\Fp>}^{\bU_{1,\calC}^F \rtimes \spann<\Fp>} (\w\psi_u)\right )(F_p)= \sum_{x\in \mathbb U} \w\psi_u^0( x\Fp(x^{-1})\Fp),\]
where $\mathbb U$ is a set of representatives of the $\bU_{2,\calC}^F$-cosets in $\bU_{1,\calC}^F$ and $\w\psi_u^0$ is defined on $\bU_{1,\calC}^F\rtimes \spann<\Fp>$ by 
\[ \w\psi_u^0(g)=\begin{cases}
\w\psi_u(g)& g\in \bU_{2,\calC}^F\rtimes \spann<\Fp>,\\
0& g\in (\bU_{1,\calC}^F\rtimes \spann<\Fp>)\setminus (\bU_{2,\calC}^F\rtimes \spann<\Fp>).\\
\end{cases} \]
Let $x\in\bU_{1,\calC}^F$ with $ x\Fp(x^{-1})\Fp\in \bU_{2,\calC}^F\rtimes \spann<\Fp>$, then $x\in\bU_{2,\calC}^F\Fp (x)$.
Fix an ordering of $\Sigma^+$ such that $ \hc^{-1}(1)$ occurs at the end of this ordering. According to Proposition 8.2.1 of \cite{Springer} the element $x$ is a product of elements $x_\al(t_\al)$ with $t_\al\in\oFp$, where the product is taken in this fixed order. Since each root subgroup is $\Fp$-invariant, $x\in\bU_{2,\calC}^F\Fp (x)$ implies that $x_\al(t_\al)^\Fp=x_\al(t_\al)$ for all $\al\in \Sigma^+$ with ${ \hc(\al)= 1}$. Hence there exists some element $y\in \bU_{2,\calC}^F x$ with $ y\Fp(y^{-1}) =1$. We may assume that in such a situation $x\in\mathbb U$ is chosen to satisfy $ x\Fp(x^{-1}) =1$. 
As a consequence, we see that 
\[\left (\Ind_{\bU_{2,\calC}^F \rtimes \spann<\Fp>}^{\bU_{1,\calC}^F \rtimes \spann<\Fp>} (\w\psi_u) \right)(F_p)= |\bU_{1,\calC}^{\Fp}:\bU_{2,\calC}^{\Fp}|\,\, \,\, \w\psi_u^0(\Fp )\neq 0.\]
Accordingly the character $\psi_u'$ has an extension $\w\psi_u'\in\Irr(\bU_{1,\calC}^F \rtimes \spann<\Fp>)$ with \[\w\psi_u'(\Fp)\neq 0.\]

Since $y\in \bP_{1,\calC}^{\Fp}$ the character $(\w\psi_u')^{y\gamma}$ is a character of $\bU_{1,\calC}^F \rtimes \spann<\Fp>$ and satisfies \[ \Res^{\bU_{1,\calC}^F \rtimes \spann<\Fp>}_{\bU_{1,\calC}^F}( (\w\psi_u')^{y\gamma})=\Res^{\bU_{1,\calC}^F \rtimes \spann<\Fp>}_{\bU_{1,\calC}^F}(\w\psi_u')=\psi'_u .\]
Hence there exists some linear character $\la\in\Irr (\bU_{1,\calC}^F \rtimes \spann<\Fp> )$ with $\bU_{1,\calC}^F\leq \ker(\la)$ and $(\w\psi_u')^{y\gamma}=\la \w\psi'_u$. But since $y\gamma$ and $\Fp$ commute, $\w\psi_u'(\Fp)\neq 0$ implies that $\la (\Fp)=1$ and accordingly
$(\w\psi_u')^{y\gamma}= \w\psi_u'$. Hence 
\[ \wh\Gamma_u\deq  \Ind_{\bU_{1,\calC}^F\rtimes\spann<\Fp>}^{\bG^F\rtimes \spann<\Fp>} (\w\psi_u') \in \ZZ_{\geq 0}\Irr(\GF\rtimes \spann<\Fp>)\]
is a $\gamma$-invariant extension of $\Gamma_u$. According to Proposition \ref{prop4_7c}, $\wh\Gamma_u$ extends to $\GF\rtimes D$, which in turn ensures that $\Gamma_u$ extends to $\GF\rtimes D$. 
\end{proof}

Thanks to the properties of $\Gamma_u$ and $\w\Gamma_\calC$ described in the theorems \ref{Mult1} and \ref{GammaSta} we can now apply Proposition \ref{Imply*}.

\renewcommand{\proofname}{Proof of Theorem~\ref{GloStaA}}
\begin{proof} 
Let $\w\chi\in\Irr(\wGF)$ and let $\calC\in\Uni(\wbG)$ be a unipotent class such that $\w\chi$ is a constituent of $\w\Gamma_\calC$ with multiplicity one. Such a class $\calC$ exists according to Theorem \ref{Mult1}. Let $\Gamma_u$ be a character associated with $\calC$ as in Theorem \ref{GammaSta}. Since $\Ind_\GF^\wGF(\Gamma_u)=\w \Gamma_\calC$ the assumptions from Proposition \ref{Imply*} are satisfied for $X\deq \w\bG^F$, $Y\deq \bG^F$, $E\deq D$ and $\rho\deq \Gamma_u$. Then the character $\chi_0\in\Irr(\GF\mid\w\chi)$ that is a constituent of $\Gamma_u$ extends to $\GF\rtimes D_{\chi_0}$ and \[ (\wGF\rtimes D)_{\chi_0}=\wGF_{\chi_0}\rtimes D_{\chi_0}.\qedhere
\]\end{proof}
\renewcommand{\proofname}{Proof}
Naturally these considerations can also be made for groups $J$ with $\bG^F\leq J\leq \wbG^F$. 

\begin{rem}Let $J$ be a group with $\GF\leq J\leq \wGF$ and $\w \chi\in \Irr(\wGF)$. Then there is some $\chi_0\in\Irr(J\mid\w\chi)$ such that $(\wGF\rtimes D)_{\chi_0}= \wGF_{\chi_0}\rtimes D_{\chi_0}$ and $\chi_0$ extends to $J\rtimes D_{\chi_0}$. In the case of $J=\wGF$ this is also a consequence of \cite[4.3.2]{Bon}. Similar considerations, using also coprime action arguments, would allow to show that for $\GF=\SL_n(2^m)$ {\em every} $\chi\in\Irr(\GF)$ extends to its stabilizer in $\GF\rtimes D$.
\end{rem}

\section{Stabilizers and extensions of characters of $\norm \bG {\bS_0}^F$}\label{sec5}
\noindent
In this section we study the action of $\Aut(\GF)_{\bS_0}$ on $\Irr(\NNN_\bG(\bS_0)^F)$, where $\bS_0$ is a Sylow $\Phi_d$-torus of $(\bG,F)$ as in Sect.~\ref{not_G_wG} above. Our aim is to show that the assumptions on $\Irrl(N)$ required in Theorem \ref{thm2_2iv} hold for $N_0=\NNN_\bG(\bS_0)^F$. 

We prove the following (note that since $\bS_0$ is $F$-stable, the group $(\GF\rtimes D)_{\bS_0} $ is well-defined).

\begin{theorem}\label{thm_IrrN_autom}
Assume the notation of Sections \ref{not_bG_wbG} and \ref{not_G_wG}. Let $\ovS_0$ be a Sylow $\Phi_d$-torus of $(\bG,F)$, $N_0\deq \NNN_\bG(\ovS_0)^F$, $\w N_0\deq \NNN_\wbG(\ovS_0)^F$ and $\psi\in\Irr(\w N_0)$. There exists some $\psi_0\in \Irr(N_0\mid \psi)$ such that 
\enumroman
\begin{itemize}
\item $O_0= (\wGF\cap O_0) \rtimes (D\cap O_0)$ for $O_0\deq \GF(\wGF\rtimes D)_{\bS_0,\psi_0}$ and
\item $\psi_0$ extends to $(\GF\rtimes D)_{\bS_0,\psi_0}$.
\end{itemize}
\enumalph
\end{theorem}

The groups $\Cent_{\bG}(\ovS_0)^F$ and $\NNN_{\bG}(\ovS_0)^F$ will be studied by considering isomorphic groups $C\deq \Cent_{\bG}(\ovS)^{vF_1}$ and $N\deq \NNN_{\bG}(\ovS)^{vF_1}$ for some $v\in\bG^{\Fp}$, a Sylow $\Phi_d$-torus $\bS$ of $(\bG,vF_1)$ and $F_1:\wbG\longrightarrow \wbG$, a slightly twisted version of $F$. Proposition \ref{prop5_6} provides a link between statements on those groups and Theorem \ref{thm_IrrN_autom}. Then, working with $\Cent_{\bG}(\bS)^{vF_1}$ and $\NNN_\bG(\bS)^{vF_1}$ we extend the result from \cite{Spaeth2,Spaeth3} that maximal extendibility holds with respect to $\cent {\GF}{\bS_0}\lhd \norm \GF {\bS_0}$ for any Sylow $\Phi_d$-torus $\bS_0$ of $(\bG,F)$, see Proposition \ref{prop5_10}. This gives a parametrization of $\Irr(\norm \bG {\bS}^{vF_1})$, see Proposition \ref{prop5_11}. After studying certain characters of subgroups of the Weyl group in Proposition \ref{prop5_13} we can check the stabilizer and extendibility statements of Theorem \ref{thm_IrrN_autom}. 

\medskip\subsection{Transfer to twisted groups}\label{subsec5.1}
 \hfill\break

 \noindent Recall the notations $q=p^m$, $F\in \{F_q ,\gamma F_q\}$, $\epsilon$, $d$, $\bT$, $\w\bT$ from Sections \ref{not_bG_wbG} and \ref{not_G_wG}. Here is a first set of additional notations.

\begin{notation}\label{not_davNF1E} 
\noindent Let $d_0\geq 1$ be the integer such that $\Phi_{d_0}(-x)=\pm\Phi_d(x)$. Let $a\deq  \lfloor \frac{n}{d_0}\rfloor$. Let $v\deq \n_1\dots \n_{d_0-1}\n_{d_0+1}\dots \n_{ad_0-1}\in\NNN_\bG (\bT)^{\spann<\gamma_0 ,F_p>}$, the product where we skip the $\n_i$'s with $i\in\ZZ d_0$, so that the image of $v$ in $\NNN_\bG (\bT)/\bT\cong\Sym_n$ is a product of $a$ disjoint cycles of order $d_0$.

Let $F_1\colon\wbG\to\wbG$ be defined by $F_1= (\gamma_0)^{\frac{1-\epsilon} 2}\circ F_q $. Let $\bS$ be the Sylow $\Phi_d$-torus of $(\bT,vF_1)$, normalized by both $\gamma_0$ and $F_p$. Let $N\deq \NNN_{\bG}(\bS )^{vF_1}\leq \w N\deq \NNN_{\wbG}(\bS )^{vF_1}$.

 Let $e\deq  2^nn!$ and $E=\Cy_2\times \Cy_{2em}$. We form \[\wbG^{F^{2e}_q}\rtimes E\] where the first summand $\Cy_2$ of $E$ acts by $\spann<\gamma_0>$ and the second summand $\Cy_{2em}$ by $\spann<F_p>$. Note that this action is faithful.
Let $\wh F_1\in E$ be the element that acts by $F_1$.

\end{notation}

Note that $\bS$ is a Sylow $\Phi_d$-torus of $(\bG,vF_1)$, see Lemma 3.3 of \cite{Spaeth3}.

\begin{prop} \label{prop5_6}
Suppose that for every $\chi\in\Irr(\w N)$ there exists some $\chi_0\in\Irr(N\mid \chi)$ such that 
\enumroman\begin{enumerate}
\item $(\w N \rtimes E)_{\chi_0} =\w N_{\chi_0} \rtimes E_{\chi_0}$, 
\item $\chi_0$ has an extension $\w \chi_0\in \Irr(N\rtimes E_{\chi_0})$ with $v\wh F_1\in\ker(\w \chi_0)$. 
\end{enumerate}\enumalph
Then Theorem \ref{thm_IrrN_autom} holds. 
\end{prop}
\begin{proof}
In a first step we construct an isomorphism. By the theorem of Lang-Steinberg, there is $g\in \bG$ such that $gF_1(g)^{-1}=v$. Then $g\in\bG^{F_q^{2e}}$. Indeed, if $e'$ is the integer with $F_1^{e'}=F_q^{2e}$, then 
\[gF_q^{2e}(g)^{-1}= g F_1(g)^{-1} F_1(g) F_q^{2e}(g)^{-1}=
v F_1(v) \cdots F_1^{e'-1}(v)= v^{e'}=1.\]
 Let 
\[ \iota:  \wbG^{F_q^{2e}}\rtimes E \longrightarrow  \wbG^{F_q^{2e}} \rtimes E \text{ with }x\longmapsto x^{g^{-1}} .\]

By definition we have $\iota (\wh F_1)=v\wh F_1$ and therefore $\iota(\GFeins)= \iota ((\bG^{F_q^{2e}} )^{F_1})=(\bG^{F_q^{2e}} )^{vF_1} = \bGvFeins$ and also $\iota(\wbGF)= \wbGvFeins$. Since $v$ is $E$-fixed, every $g'\in\{\Fp(g),\gamma_0(g)\}$ satisfies $ g'(F_1(g'))^{-1}=v$. This implies that $g^{-1}\Fp(g), g^{-1}\gamma_0(g)\in\GFeins$ and $\iota(\GF\rtimes E)= \bGvFeins\rtimes E$. 

For $D_1\deq \spann<\Fp,\gamma_0>\leq \Aut(\wGFeins)$ there exists an epimorphism $\epsilon_1:\GFeins\rtimes E \longrightarrow \GFeins\rtimes D_1$ with kernel $\spann<\wh F_1>$.

Let $N_1\deq \iota^{-1}(N)$, $\w N_1\deq \iota^{-1}(\w N)$ and $\bS_1\deq \bS^{g^{}}$. Then $\bS_1$ is a Sylow $\Phi_d$-torus of $(\bG,F_1)$, $N_1=\NNN_\bG(\bS_1)^{F_1}$ and $\w N_1=\NNN_\wbG(\bS_1)^{F_1}$. The equality $(\bGvFeins \rtimes E)_\bS =(N\rtimes E) $ implies $(\GFeins\rtimes E)_{\bS_1}= \iota^{-1} (N\rtimes E)$. 

Our assumption can be restated as saying that for every $\chi\in\Irr(\w N)$ there exists some $\chi_0\in\Irr(N \mid \chi)$ such that
\enumroman\begin{enumerate}
\item $(\wbG^{vF_1} \rtimes E)_{\bS ,\chi_0} =(\wbG^{vF_1})_{\bS ,\chi_0} (\bG^{vF_1} \rtimes E)_{\bS ,\chi_0}$, 
\item $\chi_0$ has an extension $\w \chi_0\in \Irr((\bG^{vF_1} \rtimes E)_{\bS ,\chi_0})$ with $v\wh F_1\in\ker(\w \chi_0)$. 
\end{enumerate}\enumalph
Now we apply $\iota^{-1}$ to our groups and obtain the following: for a given $\chi\in\Irr(\w N_1)$ there exists some $\chi_0\in\Irr(N_1\mid \chi)$ with $(\wGFeins \rtimes D_1)_{\bS_1,\chi_0} =(\wGFeins)_{\bS_1,\chi_0} (\GFeins \rtimes D_1)_{\bS_1,\chi_0}$, and $\chi_0$ has an extension $\w \chi_0\in \Irr((\GFeins \rtimes D_1)_{\bS_1,\chi_0})$. 
Then $\chi_0$ satisfies 
\[\GFeins(\wGFeins \rtimes D_1)_{\bS_1,\chi_0} =(\GFeins(\wGFeins)_{\bS_1,\chi_0}) (\GFeins(\GFeins \rtimes D_1)_{\bS_1,\chi_0}).\]
Recall that $\gamma_0$ and $v_0 \gamma$ induce the same automorphism on $\GFeins$, and hence
\begin{align}\label{eqS}\GFeins(\wGFeins \rtimes D)_{\bS_1,\chi_0} =(\GFeins(\wGFeins)_{\bS_1,\chi_0}) (\GFeins(\GFeins \rtimes D)_{\bS_1,\chi_0}).\end{align}
If $(\GFeins\rtimes D)_{\bS_0,\chi_0}/N$ is cyclic, then $\chi_0$ extends to $(\GFeins\rtimes D)_{\bS_0,\chi_0}$. Otherwise, since every non-cyclic subgroup of $D$ contains $\gamma_0$ there exists some $t\in\GFeins$ such that $\chi_0$ is ${t\gamma_0}$-invariant, and a $t\gamma_0$-invariant extension $\wh\chi_0$ of $\chi_0$ to $(\GFeins\rtimes \spann<\Fp>)_{\bS_0,\chi_0}$. Then $\wh\chi_0$ is $tv_0\gamma$-invariant as well and $\wh\chi_0$ then extends to $(\GFeins\rtimes D)_{\bS_1,\chi_0}$. In both cases, $\chi_0$ extends to $(\GFeins\rtimes D)_{\bS_1,\chi_0}$.

If $\epsilon=1$ or equivalently $F=F_1$, then $\bG^{F_1}=\GF$ and we can choose $\bS_0$ to be the Sylow $\Phi_d$-torus $\bS_1$ of $(\bG,F_1)$, since all Sylow $\Phi_d$-tori are $\GFeins$-conjugate. Then (\ref{eqS}) above and the extension property just proved imply the requirements of Theorem \ref{thm_IrrN_autom}. 

If $\epsilon =-1$, then $F_1=v_0F=\gamma_0F_p^m$ and $D$ is cyclic. So only the stabilizer requirement of Theorem \ref{thm_IrrN_autom} has to be checked. By the theorem of Lang-Steinberg, there is some $h\in\bG$ such that $h^{-1}F(h)=v_0$. Then as in the beginning of this proof, $x\mapsto x^{h^{-1}}$ induces isomorphisms $\bG^{F_1}\to \bG^F$, $\wbG^{F_1}\to \wbG^F$, sending $D=\spann<F_p>\leq \Aut (\wbG^{F_1})$ to the subgroup of $\Aut (\bG^F)$ generated by $\mathrm{int}(hF_p(h)^{-1})\circ F_p$ with $hF_p(h)^{-1}\in\bG^F$ (recall that $\mathrm{int}(x)$ denotes the inner automorphism $y\mapsto xyx^{-1}$). Then (\ref{eqS}) above implies the stabilizer statement we need. 
\end{proof}

\medskip\subsection{Normalizers and centralizers of Sylow $\Phi_d$-tori}\label{subsec5.2}
\hfill\break

\noindent We start with introducing more notations.

\begin{notation}\label{not_CHdVd} 
We denote $\w C\deq \Cent_{\wbG}(\bS)^{vF_1}\geq C\deq \bG\cap \w C$. Let $V=\spann<\n_i\mid i=1,\dots ,n-1>\leq \bG^{\spann<\gamma_0,F_p>}$ be the subgroup of $\SL_n(\oFp )$ of monomial matrices with coefficients in $\{-1,0,1\}$. Let $H\deq \bT\cap V$. Let 
$V_0\deq \spann<\n_i\mid i=1,\dots ,ad_0-1>$. Note that $ v\in V_0$. Let $V_d\deq \Cent_{V_0}(v)\leq H_d\deq \bT\cap V_d$. Let $W_d\deq N/C$ and $\rho\colon \NNN_\bG (\bT)\to\Sym_n$ the usual epimorphism.

\end{notation}

\begin{prop}\label{prop5_5}  \begin{enumerate}
\item \label{prop5_5a} $\rho (V_d)=\Cent_{\Sym_{ad_0}}(\rho (v))\cong \Cy_{d_0}\wr\Sym_a$, and $V_dC=N$ with $V_d\cap C = H_d$.

\item \label{prop5_5b}
Any $\xi\in\Irr(H_d)$ extends to its stabilizer in $V_d$.
\end{enumerate}
\end{prop}

\begin{proof} For (a) it is sufficient to prove the first equality since \[\NNN_\bG (\bS )^{vF_1} /\Cent_\bG (\bS )^{vF_1} = (\NNN_\bG (\bS )/\Cent_\bG (\bS ))^{vF_1} \cong \Cent_{\Sym_{ad_0}}(\rho (v))\] by the natural maps.

We now prove $\rho (V_d)=\Cent_{\Sym_{ad_0}}(\rho (v))$ and (b). 
We may assume $n=ad_0$. We also assume that $q$ is odd since otherwise $H$ is trivial. Note that $\rho (V_d)=\Cent_{\Sym_{ad_0}}(\rho (v))$ is equivalent to $|V_d: H_d|=|\Cy_{d_0}\wr\Sym_a|$.

 Let $\w V\cong \Cy_2\wr\Sym_n$ denote the group of monomial matrices in $\w\bG$ with coefficients in $ \{-1,0,1\}$, then $V =V_0=\bG\cap\w V$.

We have $v=\diag ({v'},\dots ,{v'})$ ($a$ blocks) where ${v'}$ is the $d_0\times d_0$ matrix sending the $i$-th element $e_i$ of the canonical basis of $(\oFp)^{d_0}$ to $-e_{i+1}$ for $i<d_0$ and $e_{d_0}$ to $e_1$. Then ${v'}^{d_0}=(-1)^{d_0-1}\id_{d_0}$ and therefore $v ^{d_0}=(-1)^{d_0-1}\id_{n}$.

We have $\Cent_{\Sym_n}(\rho (v))=\Cy_{d_0}\wr\Sym_a$. 
Moreover $H_d\cong (\Cy_2)^{a'}$ with $a'=a$ when $d_0$ is even, $a'=a-1$ when $d_0$ is odd. Also, $H_d$ is the group of diagonal matrices $\diag (\epsilon_1,\dots ,\epsilon_a )$ with $\epsilon_i =\pm\id_{d_0}$ and $\Pi_i\det (\epsilon_i)=1$.

Assume $d_0$ is even. Then the permutation matrices corresponding to the elements of the above complement $\Sym_a$ are in $\bG$ while ${v'}^{d_0}=-\id_{d_0}$. Then each \[v_i\deq \diag (\id_{d_0},\dots \id_{d_0},{v'},\id_{d_0}\dots )\ \  \ \ \ \ ({v'} \text{ at the } i\text{-th block)}\] satisfies $v_i^{d_0}= \diag (\id_{d_0},\dots \id_{d_0}, -\id_{d_0}, \id_{d_0}\dots )$. The latter, for $i=1,\dots ,a$, generate $H_d$. So $V_d$ projects on $\Cy_{2d_0}\wr\Sym_a$ with $H_d$ corresponding to the involutions of the base group $(\Cy_{2d_0})^a$. This gives (a) and $V_d\cong\Cy_{2d_0}\wr\Sym_a$. Any character of the base group extends to its stabilizer in $\Cy_{2d_0}\wr\Sym_a$ by \cite[8.2]{Serre}, so it suffices to show that any character of $(\Cy_2)^a$ extends to a character of $(\Cy_{2d_0})^a$ with same stabilizer in $\Sym_a$. This is easily done by choosing an extension for the two characters of the group $\Cy_2$ in the inclusion $\Cy_2\leq\Cy_{2d_0}$, then applying the same at each summand of the product $(\Cy_{2d_0})^a$. This gives (b) in that case.

Assume $d_0$ is odd. Then each $v_i\deq \diag (\id_{d_0},\dots \id_{d_0},{v'},\id_{d_0}\dots )$ (${v'}$ at the $i$-th block) has order $d_0$. Then $\w V_d\deq \Cent_{\w V}(v)\cong (\Cy_2\times\Cy_{d_0})\wr\Sym_a$ where $V_d=\Cent_ V(v)$ is the kernel of the map $(x,\si )\mapsto  \beta (x) \sgn (\si )$ where $\beta :(\Cy_2)^a\times (\Cy_{d_0})^a\to\CC^\times$ is a linear character trivial on each summand $\Cy_{d_0}$ and faithful on each summand $\Cy_2$. Denote $\w H_d\deq  \Cent_{\w V}(v)\cap\w\bT \cong (\Cy_2)^a$, so that $H_d = \w H_d\cap V$ corresponds with $(\Cy_2)^a\cap\ker(\beta )$. This gives (a). For any element of $\Irr (\w H_d)$ we define an extension to its stabilizer in $\w V_d$ as follows. Each character $\w\xi\in\Irr (\w H_d)$ defines a decomposition $\w H_d = T_+\times T_-$ where $T_+\cong (\Cy_2)^{a_+}$, respectively $T_-\cong (\Cy_2)^{a_-}$, is the product of the components of $(\Cy_2)^a$ where $\w\xi$ is trivial, respectively faithful. Then $(\w V_d)_{\w\xi} =V_+\times V_-$ with $V_+\cong (\Cy_2\times \Cy_{d_0})\wr\Sym_{a_+}$, $V_-\cong (\Cy_2\times \Cy_{d_0})\wr\Sym_{a_-}$. Let us extend $\w\xi = 1_{T_+}\times \Res^{\w V_d}_{T_-}\beta$ into $\La (\w\xi )\deq 1_{V_+}\times \Res^{\w V_d}_{V_-}\beta$. This map $\w\xi\mapsto \La (\w\xi)$ satisfies \begin{align}\label{eq_Lam1}\La (\w\xi )\Res^{\w V_d}_{(\w V_d)_{\w\xi}}\beta = \La (\w\xi \ \Res^{\w V_d}_{(\w H_d) }\beta )\end{align}  for any $\w\xi\in\Irr (\w H_d)$ (multiplying with $\beta$ just exchanges the subgroups $T_+$ and $T_-$). Moreover \begin{align}\label{eq_Lam2}\La (\w\xi )^x=\La (\w\xi ^x)\ \text{for any}\ \w\xi \in\Irr (\w H_d),\ {\text{and  }} x\in\w V_d\end{align} (evident for $x\in (\w V_d)_{\w\xi}$ , easy for $x$ in the complement isomorphic to $ \Sym_a$).
  
Let's take now $\xi\in\Irr (H_d)$. It extends to some $\w\xi\in\Irr (\w H_d)$ since $\w H_d$ is abelian. The stabilizer $(V_d)_\xi$ sends $\w\xi$ to a character with same restriction to $H_d$ hence in $\{\w\xi ,\w\xi\, \Res^{\w V_d}_{H_d}\beta\}$ so we get a group morphism $(V_d)_\xi\to \Cy_2$ by $x\mapsto \w\xi^x\ \w\xi^{-1}$ with kernel $(V_d)_{\w\xi}$ which is therefore normal in $(V_d)_\xi$ with cyclic quotient. On the other hand (\ref{eq_Lam1}) and (\ref{eq_Lam2}) above imply that $\Res^{(\w V_d)_{\w\xi}}_{(V_d)_{\w\xi}} \bigl(\La (\w\xi )\bigr)$ is $(V_d)_\xi$-stable. This character then extends to $(V_d)_\xi$ since the quotient is cyclic (see \cite[11.22]{Isa}). This is the sought extension of $\xi$.
This completes the proof of (b).
\end{proof}

\medskip\subsection{Structure of $C$}\label{ssec5_10}
\hfill\break

\noindent In order to deduce from the above some properties of $\Irr (N)$, we first need to know a little more about $C =\Cent_\bG (\bS )^{vF_1}$ and $\w C =\Cent_\wbG (\bS )^{vF_1}$. Note the use of Theorem~\ref{GloStaA}.

\begin{prop}\label{Struct_C} Denote $\bC\deq  \Cent_\bG (\bS )$, $\bG_2=[\bC ,\bC]$.\begin{enumerate}
\item \label{eq_cent}
$\w C=C.\Cent_{\wGF}(C)$ and therefore $\w C$ acts trivially on $\Irr (C)$.

\item \label{def_calgG} There exists a representative system $\calG$ of the $C$-orbits on $\Irr(\bG_2^{vF_1})$ such that every $\theta\in\calG$ satisfies
\[ \left (C \rtimes E\right )_\theta= C_\theta \rtimes E_\theta\]
and $\theta$ has a $\gamma_0$-invariant extension to $\bG_2^{vF_1}\rtimes \spann<\Fp>_\theta$ (seen as subgroup of $C\rtimes E$) whenever $\theta^{\gamma_0}=\theta$. 

\end{enumerate}\end{prop}

\begin{proof} We have $\Cent_\wbG (\bS )\cong (\oFp^\times)^{ad_0}\times \GL_{n-ad_0}(\oFp )$, so that $\bC$ has connected center isomorphic to $ (\oFp^\times)^{ad_0}$. Then the theorem of Lang-Steinberg implies that $\w C/\bigl( C.\Cent_{\wGF}(C)\bigr)$ is an image of $\Zent (\bC)/(1-vF_1)\Zent (\bC)$ which is trivial. Whence (a).

Denote $\w\bC = \Cent_\wbG(\bS)$. From the above description, we have $[[\bC ,\bC ],V_d]=1$, $\w\bC^{vF_1}\cong  \GL_{n-ad_0}(\epsilon q)\times (\Cy_{q^{d_0}-\epsilon^{d_0}} )^{a }$, and $[\bC ,\bC ]^{vF_1} = [\bC ,\bC]^{F_1}\cong \SL_{n-ad_0}(\epsilon q)$. We see first that $\Cent_{\w\bC^{vF_1}}([\bC ,\bC ]^{vF_1}) = \Zent ({\w\bC^{vF_1}}) = \Zent ({\w\bC})^{vF_1}$. Moreover (a) implies that
$\bC^{vF_1}$ induces on $[\bC ,\bC ]^{vF_1}$ the action of $\w\bC^{vF_1}$. 

So the first claim of (b) reduces to show that any character of $[\bC ,\bC ]^{vF_1}$ has a $\w\bC^{vF_1}$-conjugate $\theta$ whose stabilizer in $\w\bC^{vF_1} E$ is of the form $\w\bC^{vF_1}_\theta E_\theta$. From the explicit description of the groups above one may replace $[\bC ,\bC ]^{vF_1}$ by $\SL_{n-ad_0}(\epsilon q)$ and $\w\bC^{vF_1}$ by $\GL_{n-ad_0}(\epsilon q)$. Then our claim is contained in the stabilizer statement of Theorem~\ref{GloStaA} with $\gamma$ replaced by $\gamma_0$. But this is indeed a consequence of the actual version of Theorem~\ref{GloStaA} since $\gamma$ and $\gamma_0$ induce the same outer automorphism of $\bG^F$. The second half of (b) is the second statement in Theorem~\ref{GloStaA} with the same remark on replacing $\gamma$ by $\gamma_0$ in Theorem~\ref{GloStaA}.
\end{proof}

\medskip\subsection{Parametrization of $\Irr(N)$}
 \label{subsec5.4}
 \hfill\break
 
 \noindent
For the computation of the groups $\w N_\chi$ we first parametrize the characters of $\Irr(N)$. This is done using a so-called extension map with respect to $C\lhd N$. Such a map is defined by the following. 
\begin{defi}\label{def5_2}
Let $Y\lhd X$ be finite groups. We say that {\it maximal extendibility holds with respect to $Y\lhd X$} if every $\chi \in \Irr(Y)$ extends (as irreducible character) to $X_\chi$. 

Then, an {\it extension map with respect to $Y\lhd X$} is a map 
\[\Lambda: \Irr(Y) \rightarrow \bigcup_{Y\leq I\leq X} \Irr(I),\]
such that for every $\chi\in \calN$ the character $\Lambda(\chi)\in \Irr(X_\chi\mid \chi)$ is an extension of $\chi$.\end{defi}

Like in \cite{Spaeth1} we deduce the maximal extendibility with respect to $C\lhd N$ from the fact that maximal extendibility holds with respect to $H_d\lhd V_d$. 
Since we frequently use the following result from Lemma 4.1 of \cite{Spaeth3} we recall its statement. 
\begin{lem}\label{lem42spaeth3}
Let $Y\lhd X$ be finite groups.
\begin{enumerate}
\item \label{lem42spaeth3a} Assume $U\leq X$ with $YU=X$.
Let $\phi\in\Irr(Y)$ such that $\phi_0\deq \Res^Y_{Y\cap U}(\phi)$ is irreducible. 
For every extension $\w\phi_0$ of $\phi_0$ to $U_{\phi_0}$ there exists a unique extension $\w\phi$ of $\phi$ to $X_\phi$ with 
\[\Res^{X_\phi}_{U_\phi} (\w\phi)=\Res^{U_{\phi_0}}_{U_\phi} (\w\phi_0 ).\]
\item \label{lem5_8c}
Suppose $X\lhd \w X$ such that $Y\lhd \w X$ and $\w X=\spann<X,t>$ for some $t\in \w X$. Assume that $ t$ acts trivially on $X/Y$. Let $\phi \in\Irr(Y)$ be a $t$-invariant character and assume there is $\w \phi\in\Irr(X_\phi)$ extending $\phi$. The character $\nu\in\Irr(X_\phi)$ with $\w\phi^t=\w \phi \nu$ satisfies $\ker(\nu)=X_{\wh \phi}$, where $\wh\phi$ is an extension of $\phi$ to $\spann<Y,t>$. 
\end{enumerate}
\end{lem}
\begin{proof}
The first part follows from \cite[4.1]{Spaeth3}. Note that the statement there implies only that an extension $\w\phi$ exists, but according to \cite[6.17]{Isa} the claimed property defines $\w\phi$ uniquely.
 
For the proof of part (b) we see that for every $x\in\ker(\nu)$ the element $t$ normalizes $\spann<Y,x>$ since $t$ acts trivially on $X/Y$. Since $\nu (x)=1$, the irreducible character $\Res^{X_\phi}_{\spann<Y,x>}(\w \phi)$ is $t$-invariant. This proves that $\Res^{X_\phi}_{\spann<Y,x>}(\w \phi)$ extends to $\spann<Y,x,t>$, and $x\in X_{\wh \phi}$.

On the other hand for every $x\in X_{\wh \phi}$ we have $\spann<Y,t>\lhd \spann<Y,x,t>$ since $\spann<Y,x,t>/Y$ is abelian.
Because of  $\spann<Y,x,t>=\spann<Y,x> \spann<Y,t>$ we can apply (a) and  hence $\phi$ extends to some $\kappa \in\Irr(\spann<Y,x,t>)$ with $ \Res^{\spann<Y,x,t>}_{\spann<Y,x>} (\kappa)=\Res^{X_\phi}_{\spann<Y,x>} (\w \phi)$. This implies $\Res^{X_\phi}_{\spann<Y,x>} (\w \phi^x)=\Res^{X_\phi}_{\spann<Y,x>} (\w \phi)$. By the definition of $\nu$ we see that $x\in\ker(\nu)$.
\end{proof}

In order to control the action of automorphisms on $\Irr(N)$ we construct an extension map with respect to $C\lhd N$ with additional equivariance properties. Recall $E$ and its action from Notation~\ref{not_davNF1E}.

\begin{prop}\label{prop5_10} 
There exists an extension map $\Lambda$ with respect to $C\lhd N$ such that 
\enumroman \begin{enumerate}
\item $\Lambda$ is $N\rtimes E$-equivariant. 
\item \label{prop5_10iii}
For every $\xi\in\Irr(C)$ the character $\Lambda(\xi)$ has an extension $\wh \xi\in \Irr\left (\spann<N\rtimes E>_\xi\right )$ with $v\wh F_1\in \ker(\wh\xi)$. 
\item \label{prop5_9iii}
For every $t\in \wbTvF$ and $\xi\in\Irr(C)$ there exists a linear character $\nu\in\Irr(N_\xi)$ with $\La(\xi)^t=\La(\xi)\nu$,
and $\nu$ is the lift of a faithful character of $N_\xi/N_\wxi$ for any $\wxi\in\Irr(\spann<C,t>\mid \xi)$.
\end{enumerate}\enumalph
\end{prop}

\begin{proof} 
 According to Lemma \ref{lem5_8c}, part (iii) is satisfied for any extension map $\La$ with respect to $C\lhd N$, since $\w \bT^{vF_1}$ stabilizes all elements of $\Irr (C)$ because of Proposition~\ref{eq_cent} and $\wbT^{vF_1}$ acts trivially on \[N/C= {\mathrm N}_ \bG ( \bS )^{vF_1}/{\mathrm C}_
\bG ( \bS )^{vF_1}= ({\NNN}_ \bG ( \bS ) /{\Cent}_ \bG ( \bS
))^{vF_1} =({\NNN}_\wbG ( \bS ) /{\Cent}_\wbG ( \bS ))^{vF_1} .\]

For the proof of (i) and (ii), it suffices to show that every
$\xi\in\Irr (C)$ extends to its stabilizer in $NE/\spann<vF_1>$. More precisely one defines first $\Lambda$ on an $N\rtimes E$-transversal in $\Irr(C)$ and takes the extension of $\xi$ to $N_\xi$ to be the restriction of an extension of $\xi$ to $(N\rtimes E)_\xi/\spann<vF_1>$. Then the remaining values of $\Lambda$ can be constructed using $N\rtimes E$-conjugation. 

Denote $\bT_1\deq \{\diag(t_1,\ldots,t_n)\in \bG\mid t_i=1 \text{ for } ad_0<i\leq n\}$, $G_2\deq  [ \bC , \bC]^{vF_1}$, $T_1\deq  \bT^{vF_1}_1$, and note that $C/ (T_1\times G_2) $ is cyclic. Let $\xi\in \Irr (C)$. By Proposition \ref{def_calgG} there exists some $\xi_0\in\Irr (T_1\times G_2\mid\xi )$ such that $\xi_0=\la\times\theta$ with $\theta\in {\mathcal G}$ and $\la\in\Irr (T_1 )$. Denote $\la_0 =\Res_{H_d}^{T_1} (\la )$ . Let moreover
$\w\xi_0\in\Irr (C_{\xi_0})$ be the extension of $\xi_0$ with 
\begin{align}\label{eq_def_wxi0}
\xi &= \Ind^C_{C_{\xi_0}}(\w\xi_0).
\end{align}

Since $N=CV_d$ with $[V_d,G_2 ]=1$, and using (\ref{eq_def_wxi0}), it is clear that 
\begin{align}\label{eq_stab_xi0}
N_{\w\xi_0} = C_{\w\xi_0}(V_d)_{\w\xi_0} \text{ and }& N_\xi =C
N_{\w\xi_0}=C (V_d)_{\w\xi_0}.
\end{align}
From the properties of $\calG$ described in Proposition \ref{def_calgG} we see that 
\begin{align}\label{eq_stab_xi}
 (NE)_{\w\xi_0} =
C_{\w\xi_0}(V_d E)_{\w\xi_0} \ \text{and}\ (NE)_\xi =C
(NE)_{\w\xi_0}= C_{\w\xi_0}(V_d E)_{\w\xi_0}
\end{align}
Again by Proposition~\ref{def_calgG}, $\theta$ has an extension $ \theta'\in \Irr((G_2 \rtimes E)_\theta)$ with $v\wh F_1\in\ker( \theta')$. Since $V_d$ commutes with both $E$ and $G_2$, the group $\left ( (V_d \times G_2) \rtimes E_\theta\right ) /V_d$ is well-defined and isomorphic to $G_2 \rtimes E_\theta$. Hence $\theta'$ lifts to some $\widehat\theta\in\Irr((V_d\times G_2) \rtimes E_\theta)$ with $\spann<V_d,\wh F_1> \in\ker(\wh\theta)$.

From Proposition \ref{prop5_5b}, we know that $\la_0$ extends to some linear $\w \la_0\in\Irr((V_d)_{\la_0})$. Again since $[G_2\rtimes E,V_d]=1$ there exists an extension $\wh\la_0\in \Irr(( 
(V_d)_{\la_0}\times G_2)\rtimes E)$ of $\w\la_0$ with $\spann<G_2,v\wh F_1>\in\ker(\wh\la_0)$.

Now let
\[ \psi_0\deq  \Res^{((V_d)_{\la_0}\times G_2)\rtimes E }_{((V_d\times
G_2)\rtimes E)_{\la_0\times\theta }}( \widehat\la_0) \,\,\,\, \Res^{(V_d\times
G_2)\rtimes E_\theta}_{((V_d\times G_2)\rtimes E)_{\la_0\times\theta }}
(\widehat\theta).\]
Then $\psi_0\in\Irr ( (V_d\times G_2)\rtimes E_{\la_0\times\theta })$ and $v\wh F_1\in \ker(\psi_0)$. Moreover $\psi_0$ is an extension of $\la_0\times\theta$.

According to Lemma~\ref{lem42spaeth3a} there exists an extension
$\psi\in\Irr ((C_{\w\xi_0}V_d\rtimes E)_{\w\xi_0})$ of $\w\xi_0$ with $\Res^{(V_d\times G_2)\rtimes E_{\la_0\times\theta }}_ {(V_d\times G_2)\rtimes E_{\w\xi_0}} (\psi_0)= \Res^{(C_{\w\xi_0}V_d\rtimes E)_{\w\xi_0}}_{(V_d\times G_2)\rtimes E_{\w\xi_0}} (\psi)$. Note that we apply the lemma with $X\deq (C_{\wxi_0 }V_d)\rtimes E$, $Y\deq C_{\wxi_0}$, $U\deq  (V_d\times G_2) \rtimes E$, $\phi \deq \w\xi_0$, and $\w\phi_0\deq \psi_0$. This is possible since with this choice we have $Y\cap U=H_d \times G_2$, $\Res^Y_{Y\cap U}(\phi) =\phi_0=\la_0\times \theta$ and also $YU=X$. Now, according to (\ref{eq_stab_xi}) we have $(NE)_{\w\xi_0}=(C_{\w\xi_0}V_d\rtimes E)_{\w\xi_0}$, so that $\psi\in\Irr ((N\rtimes E)_{\w\xi_0})$. Moreover $v\wh F_1\in \ker(\psi_0)\cap ((V_d\times G_2)\rtimes E_{\w\xi_0})\leq \ker(\psi) $.

Now the character $\Ind^{(NE)_\xi}_{(NE)_{\w\xi_0}} (\psi )$ is an extension of $\xi=\Ind^C_{C_{\w\xi_0}}(\w\xi_0)$ thanks to (\ref{eq_stab_xi}) above.\end{proof}

The above extension map enables us to give a parametrization of $\Irr(N)$ by standard Clifford theory. If $N'/C$ is a subgroup of $N/C$ and $\eta\in\Irr (N'/C)$ we also use the same letter for the corresponding character of $N'$ with $C $ in its kernel.

\begin{prop}\label{prop5_11}
Let $\Lambda$ be the extension map from Proposition \ref{prop5_10} with respect to $C\lhd N$. Let $\calP$ be the set of pairs $(\xi,\eta)$ with $\xi\in\Irr(C)$ and $\eta\in\Irr(W_\xi)$, where $W_\xi\deq N_\xi/C$. Then 
\[\Pi:\calP\longrightarrow\Irr(N) \text{ defined by } (\xi,\eta)\longmapsto \Ind_{N_\xi}^{N}(\La(\xi)\eta),\]
is surjective and satisfies
\enumroman\begin{enumerate}
\item $\Pi(\xi,\eta)=\Pi(\xi',\eta')$ if and only if there exists some $n\in N$ such that $\xi^n=\xi'$ and $\eta^n=\eta'$,
\item $\Pi(\xi,\eta)^\si=\Pi(\xi^\si,\eta^\si)$ for every $\si\in E$.
\item Let $t\in\wbT^{vF_1}$. Then $\Pi(\xi,\eta)^t=\Pi(\xi,\eta\nu)$, where $\nu\in\Irr(N_\xi)$ is given by $\La(\xi)^t=\La(\xi)\nu$ and is a linear character with $N_{\wxi}=\ker(\nu)$ for any extension $\wxi\in \Irr(\spann<C,t>)$ of $\xi$. 
\end{enumerate}\enumalph
\end{prop}

\begin{proof} Let $\chi\in\Irr(N)$. By Clifford theory the set $\Irr(C\mid \chi)$ forms an $N$-orbit. For any $\xi\in\Irr(C\mid\chi)$ there exists a unique $\eta\in\Irr(W_\xi)$ with $\chi=\Ind_{N_\xi}^{N}(\La(\xi)\eta)$. Hence $\Pi$ is surjective. Proposition \ref{prop5_10} implies the remaining properties: Since $\Lambda$ is an $N$-equivariant extension map we have 
\[ \Ind_{N_\xi}^N\left (\La(\xi)\eta\right )= \Ind_{N_{\xi^n}}^N\left (\La(\xi^n)\eta^n\right )\]
for every $n\in N$. This completes the proof of (i). Since $\La$ is also $E$-equivariant, (ii) holds. Part (iii) is a consequence of Proposition~\ref{prop5_9iii}.
\end{proof}
\medskip

\begin{rem}\label{rem5.12}The property in (iii) implies that for a character $\xi\in\Irr(C)$, $\w\xi\in\Irr(\w C\mid \xi)$ and $\eta_0\in\Irr(W_\wxi)$ the set $ \Set{(\xi,\eta')| \eta'\in\Irr(W_\xi\mid \eta_0)}$ corresponds via $\Pi$ to a $\wbT^{vF_1}$-orbit. Because of $\w N=N\wbT^{vF_1} $ this is then also a $\w C$-orbit.\end{rem}

 \medskip\subsection{Stabilizers and extendibility of characters of $N$}\label{subsec5.5}
 
\begin{prop}\label{prop5_13}
Let $\xi\in \Irr(C)$, $\w\xi\in\Irr(\w C\mid \xi)$, $W_d\deq N/C$, $W_\xi=N_\xi/C$, $W_\wxi=N_{\w\xi}/C$ and $\eta_0\in\Irr(W_{\w\xi})$. Let $K\deq \NNN_{ {W_d} }(W_\xi)\cap {\NNN_{W_d}( W_\wxi)}$.
 Then there exists a character $\eta\in\Irr(W_\xi\mid \eta_0)$ such that the following holds
 \enumroman
 \begin{enumerate}
 \item $\Set{\eta^w| w\in K}\cap \Irr(W_\xi\mid \eta_0)=\{\eta\}$, 
 \item \label{prop5_13ii} $\eta$ extends to $K_\eta$,
 \item $\eta\in\Irr(W_\xi\mid \eta_0)$ has an extension     $\wh\eta \in\Irr( K_{\eta}\times E)$ with $v\wh F_1\in\ker(\wh\eta)$.
 \end{enumerate}
\end{prop}
\begin{proof}By Proposition~\ref{prop5_5a}, $W_d\cong\Cy_{d_0}\wr \Sym_{a}$. The group $\w C$ can be written as direct product of cyclic groups of order $q^{d_0}-\epsilon^{d_0}$ and a general linear group $\GL_{n-ad_0}(\epsilon q)$. Each of the groups $ \Cy_{d_0}$ above acts on one single factor and $\Sym_a$ permutes the factors. After suitable $\w C$-conjugation every character $\wxi\in\Irr(\w C)$ has a stabilizer $W_\wxi$ with the following property: there exist positive integers $j$, $d_i$ ($1\leq i\leq j$), $a_i$ ($1\leq i\leq j$) and pairwise disjoint subsets $M_1,\ldots , M_j \subseteq \Lset{1,\ldots, a}$ associated with $\w\xi$ such that $|M_i|=a_i$ and 
\[W_{\w\xi}= \Set{ \left( (\zeta_1, \ldots , \zeta_{a}) , \si \right)\in W_d|
\si(M_i)=M_i \text{ for all } 1\leq i \leq j \text{ and } \zeta_k^{d_i}=1 \text{ for all }k\in M_i }.\]

First we write $\eta$ as induced character where we use that $W_\wxi$ is a wreath product. Let $\wh Z\deq \Lset{ \left( (\zeta_1, \ldots , \zeta_{a}) ,1\right) \in W_d}$, and $Z\deq \wh Z\cap W_\wxi$. Since $\wh Z\lhd W_d$, $Z\lhd W_\wxi$. 
Let \[ \wh Z_i\deq \Set{ \left( (\zeta_1, \ldots , \zeta_{a}) ,1\right) \in \wh Z| 
\zeta_k=1 \text{ for every } k \neq\text{ i}}\] 
and $Z_i\deq \wh Z_i\cap W_\wxi$. For every $i$ the group $Z_i$ satisfies $Z_i\cong \Cy_{d_k}$ where $k$ is defined by $i\in M_k$. 
Then $Z=Z_1\times \cdots \times Z_{a}$. Let $\nu\in\Irr(Z\mid \eta_0)$. Then $\nu=\nu_1\times \cdots \times\nu_{a}$ for some $\nu_i\in\Irr(Z_i)$. For $\nu\in\Irr(Z)$ we observe that $(W_{\wxi})_\nu=Z\rtimes S$ with
\[ S\deq \Lset{((1,\ldots, 1),\si)\in (W_{\wxi})_\nu}.\]
Here $S$ permutes the groups $Z_i$. Then $(W_{\wxi})_\nu=Z \rtimes S$, and the character $\nu$ has an extension $\w\nu\in (W_{\wxi})_\nu$ with $\w\nu(S)=1$, see \cite[\S 8.2]{Serre}. By Clifford theory there exists a unique $\kappa\in\Irr((W_{\wxi})_\nu)$ with $Z\leq \ker(\kappa)$ and \[\eta_0= \Ind_{(W_{\wxi})_\nu}^{W_\wxi}(\w\nu\kappa).\]

For every $i$ the map $\iota_i:\wh Z_i\longrightarrow \Cy_{d_0}$ given by $((1,\ldots, 1,\zeta_i, 1,\ldots, 1),1)\longmapsto \zeta_i$
is an isomorphism. We fix for every $d'\mid d_0$ an extension map $\Pi_{d'}:\Irr(\Cy_{d'})\longrightarrow \Irr(\Cy_{d_0})$. This is possible since $\Cy_{d_0}$ is abelian. Let $k$ be an integer with $1\leq k\leq j$ and $i\in M_k$. Then $\iota_i$ and $\Pi_{d_k}$ determine an extension $ \wh\nu_i\in\Irr(\wh Z_i)$ of $\nu_i$. By this definition $\wh \nu=\wh \nu_1\times \cdots \times \wh\nu_a\in \Irr(\wh Z)$ satisfies $K_\nu=K_{\wh \nu}$ where $K\deq \NNN_{ {W_d} }(W_\xi)\cap {\NNN_{W_d}( W_\wxi)}$. 

Because of $\wh Z\lhd W_d$ the group $\wh Z W_\wxi$ is well-defined and $\wh Z \lhd \wh Z W_\wxi$. Like before $\wh \nu$ has an extension $\psi\in\Irr(\wh ZS)$ with $S\leq\ker(\psi)$. 
The character $\kappa$ has an extension $\w\kappa\in\Irr(\wh ZS)$ with $\wh Z\leq\ker(\w\kappa)$. The character $\wh \eta_0\deq \Ind_{\wh Z S}^{\wh Z W_\wxi} (\psi\w\kappa)$ satisfies $\Res^{\wh ZW_\wxi}_{W_\wxi}(\wh \eta_0)=\eta_0$. By its definition $\wh \eta_0$ satisfies 
\[K_{\eta_0}\leq \NNN_{W_d}(W_{\wxi})_{\eta_0}\leq \NNN_{W_d}( \wh Z W_\wxi)_{\wh\eta_0}, \]
since $K \leq \wh K\deq \NNN_{W_d}( \wh Z W_\wxi)$. On the other hand $\wh K$ and hence $\NNN_{W_d}( \wh Z W_\wxi)_{\wh\eta_0}$ are wreath products, so $\wh \eta_0$ extends to some $\phi\in\Irr(\wh K_{\wh \eta_0})$ according to \cite[Chapter 4.3]{JamesKerber}. The character $\w\eta_0=\Res^{\wh K_{\wh \eta_0}}_{(W_\xi)_{\eta_0}}(\phi)$ is a well-defined extension of $\eta_0$, and 
$\eta\deq \Ind^{W_\xi}_{(W_\xi)_{\eta_0}} \left (\w\eta_0\right )\in\Irr(W_\xi\mid \eta_0)$. 

Now we check that $\eta$ has the claimed properties. We first prove that $\eta$ has the property described in (i). Let $w\in K_\eta$ with $\eta^w\in\Irr(W_\xi\mid \eta_0)$. By Clifford theory we can assume that $w\in K_{\eta_0}$. Because of $K_{\eta_0}\leq \wh K_{\wh \eta_0}$ and $\phi\in\Irr(\wh K_{\wh\eta_0})$ we see that $\phi^w=\phi$. 
Because of $\w\eta_0=\Res^{\wh K_{\wh \eta_0}}_{(W_\xi)_{\eta_0}}(\phi)$ this implies $\w\eta_0^w=\w\eta_0$ and $\eta^w=\eta$. This proves (i).

Because of $K_{\eta}=W_\xi K_{\eta_0}$ the following character 
\[\wh \eta\deq  \Ind _{K_{\eta_0}}^{K_\eta} \left (\Res^{\wh K_{\wh \eta_0}}_{K_{\eta_0}}(\phi)\right ) \in\Irr(K_\eta)\] is an extension of $\eta$, as required in (ii). The property claimed in (iii) follows from the fact that the order of $v$ divides the order of $\wh F_1$ and both are central in $K\times E$.
\end{proof}

Finally we can prove that the assumptions made in Proposition \ref{prop5_6} hold.
\begin{prop}\label{prop5_15}
For every $\w\chi\in\Irr(\w N)$ there exists $\chi\in\Irr(N\mid\w \chi)$ such that 
\enumroman\begin{enumerate}
\item \label{prop5_15i}
$(\w N \rtimes E)_{\chi} =\w N_{\chi} \rtimes E_{\chi}$, 
\item \label{prop5_15ii}
$\chi$ has an extension $\wh \chi\in \Irr(N\rtimes E_{\chi})$ with $v\wh F_1\in\ker(\wh \chi)$. 
\end{enumerate}\enumalph
\end{prop}
\begin{proof} Let $\chi'\in\Irr(N\mid\w \chi)$. We are going to show that there exists a $\w N$-conjugate $\chi$ of $\chi '$ satisfying (i) and (ii) above.  Using the map $\Pi$ from Proposition \ref{prop5_11} and the notations of Proposition~\ref{prop5_13}, we have $\xi\in\Irr(C)$ and $\eta'\in\Irr(W_\xi)$ with $\chi'=\Pi(\xi,\eta')$. 

Let $\wxi\in\Irr(\w C\mid \xi)$ and $\eta_0\in\Irr(W_\wxi\mid \eta')$. Let $K\deq \NNN_{W_d}(W_\xi)\cap {\NNN_{W_d}( W_\wxi)}$. By Proposition \ref{prop5_13} there exists some $\eta\in\Irr(W_\xi\mid \eta_0)$ such that $\Set{\eta^w| w\in K}\cap \Irr(W_\xi\mid \eta_0)=\{\eta\}$ and $\eta$ extends to some $\wh \eta'\in\Irr( K_\eta\times E)$ with $v\wh F_1\in\ker(\wh \eta')$. Let $\chi\deq \Pi(\xi,\eta)$. Now Remark \ref{rem5.12} implies that $\chi$ and $\chi'$ are $\w N$-conjugate.

First we verify that $\chi$ satisfies the equation from (i), namely 
\[(\w N \rtimes E)_{\chi} =\w N_{\chi} \rtimes E_{\chi}.\] 
Let $x\in (\w N \rtimes E)_\chi$. After suitable $N$-multiplication we can assume that $\xi^x=\xi$. Furthermore let $x=n t e$ with $n\in N$, $t\in\wbT^{vF_1} $ and $e\in E$. Denote $w\deq \rho (n)= nC\in W_d$. Then the character $\chi^x$ satisfies 
\[ \chi^x=\Pi(\xi^{ne}, \eta^{w} \nu),\]
for some linear $\nu\in\Irr(W_{\xi^n})$. This character satisfies according to Proposition \ref{prop5_11} the inclusion $W_\wxi\leq\ker(\nu)$, more precisely $\ker(\nu)=W_{\wxi_0}$ for $\w\xi_0\deq \Res^{\w C}_C(\w\xi)$.

Since $\xi^x=\xi$ the equation $\chi=\chi^x$ is equivalent to 
\[ \eta=\eta^{w} \nu.\]

We show that then $w$ normalizes $W_\xi$ and $W_\wxi$, hence $w\in K$. Since $\xi^x=\xi$ we have that $W_\xi^x=W_{\xi^x}=W_\xi$. Since $t$ and $e$ act trivially on $W_\xi$ this implies $W_\xi^{w}=W_\xi$. On the other hand when we consider $W_\wxi$ we have
\[ W_\wxi^{w}=W_\wxi^{x}= W_{(\wxi )^x}.\]
Since $\w\xi$ and $(\wxi )^n$ are extensions of $\xi$ there exists some linear character $\delta \in\Irr(\w C)$ with $C\leq \ker(\delta)$ and $(\wxi )^x=\wxi \delta$. Since $[\w N,\w C]\leq C$ this character $\delta$ extends to some character of $\w N$, especially $\delta$ is $\w N$-invariant. Hence $W_\wxi=W_{\wxi\delta}$. Like above this proves $w\in\NNN_{W_d}(W_\wxi)$ (recall $W_d=N/C$). We can conclude that $w\in\NNN_{W_d}(W_\wxi)\cap \NNN_{W_d}(W_\xi)$.

Since $\eta\nu^{-1}\in\Irr(W_\xi\mid \eta_0)$ and $ \Set{\eta^w| w\in K}\cap \Irr(W_\xi\mid \eta_0)=\{\eta\}$ this implies $\eta^{w}=\eta=\eta\nu$. By Proposition \ref{prop5_11} we see $\chi^t=\Pi(\xi,\eta\nu)=\Pi(\xi,\eta)=\chi$. This proves $t\in \w N_\chi$ and hence $ (\w N \times E)_\chi=\w N_\chi \rtimes E_\chi$.

Based on the above observations on $\chi$ we can easily verify that $\chi$ has also the extension described in (ii). Let $\wh N\deq N\rtimes E$ and $I\deq  (\wh N_{\xi})_{\La(\xi)\eta}$. Note that by Clifford theory $I$ satisfies $\wh N_\chi= N I$. Now, according to Proposition~\ref{prop5_10iii} there exists an extension $\wh \xi' \in \Irr(\wh N_{\xi})$ of $\La(\xi)$ with $v\wh F_1\in\ker(\wh\xi')$. Let $\wh \xi\deq \Res^{\wh N_\xi}_{(\wh N_{\xi})_{\La(\xi)\eta}}(\wh \xi')$. The character $\eta\in\Irr(W_\xi)$ has an extension to $ K_\eta\times E$. The quotient $I/C$ is a subgroup of $ K_\eta\times E$ and hence the inflation of $\eta$ to $N_\xi$ extends to some $\wh \eta\in\Irr(I)$ with $v\wh F_1\in \ker(\wh\eta)$. The character $\wh \xi \, \,\wh\eta$ is an extension of $\La(\xi)\eta$ to $I$. Hence $\wh \chi\deq \Ind_{I}^{\wh N_\chi}( \wh \xi\,\wh\eta)$ is an extension with the required properties. 
\end{proof}

\medskip\subsection{Some consequences}\label{subsec5.6} 
\hfill\break

\noindent 
The first consequence of the above is the completion of the proof of Theorem \ref{thm_IrrN_autom}. Indeed, thanks to Proposition \ref{prop5_15}, we can apply Proposition \ref{prop5_6}. This implies Theorem \ref{thm_IrrN_autom}.

For later use we also deduce the existence of another extension map with respect to certain subgroups of $\wbG^F$.
\begin{cor}\label{cor5_17}
Let $\bS_0$ be a Sylow $\Phi_d$-torus of $(\bG,F)$, $\w C\deq \Cent_{\w\bG}(\bS_0)^F$ and $\w N\deq \NNN_{\w\bG}(\bS_0)^F$. Then there exists an extension map $\w \La$ with respect to $\w C\lhd \w N$ such that 
\enumroman\begin{enumerate}
\item \label{cor5_17i}
$\w\La$ is $(\wGF\rtimes D)_{\bS_0}$-equivariant, 
\item \label{cor5_17ii} 
 $\w\La\left (\xi \,\Res_{\w C}^{\wGF}(\delta)\right )= \w\La(\xi)\,\Res_{\w N_\xi}^ {\wGF}(\delta)$ for every $\xi\in\Irr(\w C)$ and $\delta\in\Irr(\wGF)$ with $\GF\leq \ker(\delta)$.
\end{enumerate}\enumalph
\end{cor}
\begin{proof}
Let $C\deq \Cent_{\bG}(\bS_0)^F$ and $N\deq \NNN_{\bG}(\bS_0)^F$. By Proposition \ref{eq_cent}, the Sylow $\Phi_d$-torus $\bS$ of $(\bG,vF_1)$ defined there satisfies $\Cent_{\bG}(\bS)^{vF_1} \Cent_{\wbG}\left (\Cent_{\bG}(\bS)\right )^{vF_1}=\Cent_\wbG(\bS)^{vF_1}$. This implies $\w C=C\Cent_{\wGF}(C)$. Since also $\w C/C$ is cyclic, $\Res_C^{\w C} (\xi)$ is irreducible for every $\xi \in\Irr(\w C)$.

Let $\La$ be an extension map with respect to $C\lhd N$. According to Proposition \ref{prop5_10} there is an $(\bG^{vF_1}\rtimes E)_{\bS_0}$-equivariant extension map $\La_0$ with respect to $\Cent_{\wbG}(\bS_0)^{vF_1}\lhd \NNN_{\wbG}(\bS_0)^{vF_1}$. Applying the bijection $\iota$ from the proof of Proposition \ref{prop5_6}, we see that $\La$ exists and can be chosen to be $(\bG^F\rtimes D)_{\bS_0}$-equivariant. (This uses the fact that all Sylow $\Phi_d$-tori of $(\bG,F)$ are $\GF$-conjugate.) According to Lemma \ref{lem42spaeth3a}, for every character $\xi\in\Irr(\w C)$ there exists a unique extension $\w \xi\in\Irr(\w N_\xi)$ of $\xi$ with $ \Res^{\w N_\xi}_{N_\xi} (\w \xi) = \Res_{N_\xi}^{N_{\xi_0}} (\La(\xi_0))$ for $\xi_0\deq \Res^{\w C}_C(\xi)$. Let $\w\La$ be the extension map with respect to $\w C\lhd \w N$ that associates to $\xi\in\Irr(C)$ the character $\w\xi$ constructed as above. By the uniqueness of $\w\xi$, $\w\La$ is $(\wbGF\rtimes D)_{\bS_0}$-equivariant and satisfies the equation from (ii). 
\end{proof}

\section{Construction of an equivariant bijection for $\wGF$}\label{sec_Bij_Gu}

\noindent A bijection between characters of $\lp$-degree has been constructed in \cite{Ma06}. We slightly generalize this to consider characters of $\wbG^F$ that lie over characters of $\bG^F$ with $\lp$-degree. 
In the following we continue using the notation from Sect.~\ref{not_bG_wbG} and \ref{not_G_wG}, especially $\bG \deq \SL_n(\oFp )$, $\wbG \deq \GL_n(\oFp )$ and $F:\wbG\to\wbG$ is a Steinberg endomorphism defining an $\FF_q$-structure. In view of the hypothesis in Theorem~\ref{thm2_2ii}, our aim in this section is to construct the following bijection.

\begin{theorem}\label{thm6_1} Let $\ell$ be a prime number with $\ell\nmid q$, and $d$ the associated integer from Sect.~\ref{not_d}. Let $\bS_0$ be a Sylow $\Phi_d$-torus of $(\bG,F)$. Let $\w N=\NNN_{\wbG}(\bS_0)^F$ and $ N=\NNN_{\bG}(\bS_0)^F$. 

There exists a $(\wbG^F\rtimes D)_{\bS_0}$-equivariant bijection
\[\Omegau: \Irr (\wbG^F\mid \Irrl (\bG^F))\longrightarrow \Irr (\w N\mid \Irrl (N))\] such that 
\begin{itemize}
\item
$ \Omegau(\Irr (\w\bG^F | \Irr_\lp ({\bG^F}))\cap\Irr(\w\bG^F \mid \nu))\subseteq \Irr(\Nu\mid \nu)$ for every  $\nu \in \Irr(\Z(\w\bG^F))$,

\item $\Omegau(\chi\delta)= \Omegau(\chi)\, \Res^{\wbG^F}_{\w N}(\delta)$ for every $\chi , \delta\in\Irr (\w\bG^F | \Irr_\lp ({\bG^F})$ with $\bG^F\leq \ker(\delta)$.
\end{itemize}
\end{theorem}
 
Malle has established in Theorem 7.3 of \cite{Ma06} a bijection for $\ell'$-characters in groups of simply-connected type. 
For the case where the underlying algebraic group has connected center we have introduced in \cite[4.1]{CabSpaeth} a similar character correspondence using sets of parameters. Below we define a related set that parametrizes $\Irr(\wGF\mid \Irrl(\bG^F))$.

We use freely the facts and notation for groups in duality and (Lusztig's) geometric series of irreducible characters from \cite[\S 13]{DigneMichel}. 
Let $\w\bG^* $ be the dual of $\w\bG$ and $F^*:\w\bG^*\longrightarrow\w\bG^*$ the dual of $F:\wbG\longrightarrow\wbG$. 
Note that $\w\bG^*$ is isomorphic to $ \GL_n(\oFp )$. Let $\pi: \wbG^* \longrightarrow \bG^* $ be the canonical epimorphism.  For $\bC\deq \Cent_\bG (\bS_0)$ let $\bC^*$ be an $F^*$-stable Levi subgroup of $\bG^*$ dual to $\bC$ (or simply the image of $\bC$ by $\pi$ through the identification of $\wbG^*$ with $\wbG$). The Sylow $\Phi_d$-torus $\bS_0^*$ of $\Zent (\bC^*)$ is a Sylow $\Phi_d$-torus of $(\bG^*,\Fdual)$ (identified with $\pi (\bS_0)$).

For $F$-stable Levi subgroups $\bL_1\leq \bL_2$ we denote by $\R^{\bL_2}_{\bL_1}$ the Lusztig functor, see \cite[11.1]{DigneMichel}. As explained in \cite[2.7]{CabSpaeth}, we may omit mentioning the specific parabolic subgroup of $\bL_2$ having $\bL_1$ as Levi complement when this functor is applied to unipotent characters (see \cite[1.33]{genblo} and also \cite[15.7]{DigneMichel} for formulas in our case where $\wbG=\GL_n(\oFp )$).

When $s\in(\wbG^*)^\Fdual$ is a semi-simple element, one denotes by $\E (\wbG^F ,s)\subseteq\Irr (\wbG^F)$ the {\em geometric} series associated to $s$. 
According to the Jordan decomposition of characters for groups with connected center there is a bijection 
\[ \E (\Cent_{\wbG^*}(s)^\Fdual,1)\longrightarrow \E(\wbG^F ,s) \text{ with } \zeta\mapsto \chi_{s,\zeta}^{\wbG} \] 
see for instance \cite[11]{Cedric} and \cite[2.7]{CabSpaeth}. Note that for $\wbG=\GL_n(\oFp )$ we have $\chi_{s,\zeta}^{\wbG} = \R^{\wbG}_{\wbG (s)}(\hat s\zeta )$ where $\wbG (s)$ is in duality with the Levi subgroup $\Cent_{\wbG^*}(s)$ of $\wbG^*$ and $\hat s$ is the linear character of $\wbG(s)^F$ associated with $s$ by duality, see \cite[13.25]{DigneMichel} or \cite[15.10]{CabEng}.

Recall that for groups $Y\leq X$ we write ${\W}_X(Y)\deq \NNN_X(Y)/Y$.

 \begin{notation}[Sets of parameters]\label{MalSet} Let $\w\calM_G$ be the set of quadruples $(\bS^* ,s,\la ,\eta)$ where $\bS^*$ is a Sylow $\Phi_d$-torus of $(\w\bG^*,\Fdual)$, $s\in \Cent_{\w\bG^*} (\bS^*)^\Fdual$ is a semi-simple element, $\la\in\E ({\Cent_{\w\bG^*}(\bS^*,s)^\Fdual},1)$ with $\la (1)_\ell =1$, and $\eta\in\Irr _\lp (\W_{\Cent_{\w\bG^*} (s)}(\Cent_{\w\bG^*}(\bS^*,s))_\la ^\Fdual)$. Let $\w\calM_N$ be the set of triples $(s,\la ,\eta)$ that satisfy $ (\bS_0^* ,s,\la ,\eta)\in\w\calM_G$. \end{notation} 

Note that $\wbG^*{}^\Fdual$ and $\NNN_{\w\bG^*}(\bS_0^*)^\Fdual $, respectively, act naturally on $\w\calM_G$ and $\w\calM_N$, respectively. In order to associate a character with each element of $\w\calM_G$ and $\w\calM_N$ we have to recall some known facts. 

Denote $\w\bC\deq \Cent_\wbG(\bS_0)$ and $\w\bC\deq \Cent_{\wbG^*}(\bS_0^*)$. Let $(s,\la ,\eta)\in\w\calM_N$. There exists an isomorphism 
\[i_{s,\la}: \w N_{\chi^{\w\ovC}_{s,\la}} /\w\ovC^F\to \W_{\Cent_{\w\bG^*} (s)}(\Cent_{\w\bG^*}(\bS_0^*,s))_\la ^\Fdual,\] see \cite[3.3]{CabSpaeth}.
Since $\bS^*_0$ is a Sylow $\Phi_d$-torus of $(\wbG^*,\Fdual)$, the pair $({\Cent_{\w\bG^*}(\bS_0^*,s)},\la )$ is a $d$-cuspidal pair in the sense of \cite{genblo}. According to \cite[3.2]{genblo}, there exists an injective map 
 \[ \calI^{\Cent_{\w\bG^*}(s)}_{\Cent_{\w\bG^*}(\bS_0^* ,s),\la}: \Irr (\W_{\Cent_{\w\bG^*} (s)}(\Cent_{\w\bG^*}(\bS_0^*,s))_\la ^\Fdual )\longrightarrow \E ({\Cent_{\w\bG^*}(s)}^\Fdual,1), \]
  such that for every $\eta\in\Irr (\W_{\Cent_{\w\bG^*} (s)}(\Cent_{\w\bG^*}(\bS_0^*,s))_\la ^F )$ the unipotent character $\calI^{\Cent_{\w\bG^*}(s)}_{\Cent_{\w\bG^*}(\bS_0^* ,s),\la}( \eta)$ is an irreducible component of the generalized character ${\R}^{\Cent_{\w\bG^*}(s)}_{\Cent_{\w\bG^*}(\bS_0^* ,s)}(\la )$.

Let $\w N\deq \NNN_\wbG (\bS_0 )^F$. As usual $\Cent^\circ_{\bG^*}(t)$ denotes the identity component of $\Cent_{\bG^*}(t)$ for every $t\in\bG^*$. Recall that according to Corollary~\ref{cor5_17} there exists an extension map $\w\Lambda$ with respect to $\w\bC^F\lhd\w N$ with certain additional properties. 

 \begin{prop}\label{MalBij} 
 \begin{enumerate} \item The map $\psi^{(\w G)}:\w\calM_G\longrightarrow \Irr (\wbG^F)$ defined by \[ \w\calM_G\ni (\bS^*,s,\la ,\eta )\longmapsto \chi_{s,\zeta}^{\w\bG}\in\Irr (\wbG^F)\text{ with } \zeta \deq \calI^{\Cent_{\w\bG^*}(s)}_{\Cent_{\w\bG^*}(\bS^* ,s),\la}(\eta )\] induces a monomorphism from the set of $\w\bG^*{}^{F^*}$-orbits on $\w\calM_G$ into $\Irr (\wbG^F)$. 
\item The map $\psi^{(\w N)}\colon\w\calM_N\longrightarrow \Irr (\w N)$ defined by \[ \w\calM_N\ni (s,\la ,\eta )\longmapsto \Ind^{\w N}_{\w N_{\xi '}}(\w\La (\xi ')\, (\eta\circ i_{s,\la}))\in\Irr (\w N)\text{ with } \xi '= \chi_{s,\la}^{\w\ovC} \in\Irr (\w\ovC^F),\] induces a monomorphism from the set of $\NNN_{\w\bG^*}(\bS_0^*)^{F^*}$-orbits on $\w\calM_N$ into $\Irr (\w N)$. 

\item For $(s,\la ,\eta )\in\w\calM_N$ the following three conditions are equivalent
\begin{itemize}
\item $\psi^{(\w G)} (\bS^*_0,s,\la ,\eta )\in\Irr (\w\bG^F\mid\Irr_\lp (\bG^F))$, 
\item $\psi^{(\w N)} (s,\la ,\eta )\in\Irr (\w N\mid\Irr_\lp (N))$, 
\item $|\bG^*{}^{F^*}:\Cent_{\bG^*}(\pi (s))^{F^*}|_\ell = 1$ and any $x\in \Cent_{\bG^*}(\pi (s))^{F^*}_\ell$ sends $(\bS_0^*,\la ,\eta )$ to a $\Cent_{\bG^*}^\circ (\pi (s))^{F^*}$-conjugate. 
\end{itemize}

\end{enumerate} 

\end{prop} 

\begin{proof} We start by considering the case where $\ell\nmid |\wbG^F:\bG^F|$. This implies $\ell\nmid |\w N:N|$ and $\ell\nmid |\Cent_{\bG^*}(\pi (s))^{F^*}:\Cent^\circ_{\bG^*}(\pi (s))^{F^*}|$, see \cite[13.16.(i)]{CabEng} for the last statement. Moreover $\Irr (\w\bG^F | \Irr_\lp ({\bG^F})) = \Irr_\lp (\w\bG^F)$ and $\Irr (\w N |\Irr_\lp (N))=\Irr _\lp(\w N )$. Accordingly, the statements above are then a consequence of \cite[4.3, 4.5]{CabSpaeth}, since our Corollary~\ref{cor5_17} gives the extension map with the properties required there.

We may then assume $\ell\mid |\wbG^F:\bG^F|$. Note that consequently $d\in\{ 1,2\}$, the centralizer of the Sylow $\Phi_d$-torus $\bS ^*$ is a torus of $\w\bG^*$, and $\la\in \E (\Cent_{\w\bG^*}(\bS^*, s),1)$ is the trivial character. We will use the latter fact in the proof of (c).
 
The considerations from \cite[Sect. 4.1]{CabSpaeth} prove that the induced maps $\psi^{(\w G)}$ and $\psi^{(\w N)}$ are injective on $\w\bG^*{}^{F^*}$-orbits of $\w\calM_G$, respectively $\NNN_{\bG^*} (\bS_0^*)^{F^*}$-orbits of $\w\calM_N$. The only difference here is that the hypothesis $|{\wbG^*{}^{F^*}}:\Cent_{\wbG^*{}^{F^*}}(s)|_\ell =1$ is not assumed, but this hypothesis was not used in the proof of \cite[4.2]{CabSpaeth}. This gives (a) and (b).

In the next step we verify the second half of our proposition. We determine the quadruples corresponding to $\Irr (\w\bG^F\mid\Irr_\lp (\bG^F))$ and the triples corresponding to $\Irr (\w N\mid\Irr_\lp (N))$.
Let $(\bS^*_0,s,\la ,\eta )\in\w\calM_G$. Then $\psi^{(\w G)}(\bS^*_0,s,\la ,\eta ) = \chi_{s,\zeta}^{\w\bG}\in\Irr (\wbG^F)\text{ with } \zeta =\calI^{\Cent_{\w\bG^*}(s)}_{\Cent_{\w\bG^*}(\bS^*_0 ,s),\la}(\eta )$.

Recall now the Jordan decomposition of characters of $ \bG^F$ (see \cite{Lu88}, \cite{CabEng}). When $s\in(\wbG^*{})^{F^*}$ is a semi-simple element, the so-called \emph{rational} series $\E (\bG^F,[\pi (s)])\subseteq\Irr (\bG^F)$ is the set of irreducible components of restrictions to $\bG^F$ of elements of the geometric series $\E (\wbG^F,s)$. In this situation, there exists a bijection \[\Irr \big(\Cent_{\bG^*} (\pi ( s))^{F^*}\mid \E (\Cent_{\bG^*} ^\circ(\pi ( s))^{F^*},1) \big)\longrightarrow \E (\bG^F,[\pi (s)]),\ \xi\longmapsto \chi^\bG_{\pi (s),\xi}\] which satisfies \[\Res^{\w\bG^F}_{\bG^F} (\chi_{ s,\zeta}^{\w\bG})=\sum_{\zeta '\in \Irr (\Cent_{\bG^*} (\pi ( s))^{F^*}\mid \zeta )}\chi^\bG_{\pi (s),\zeta '}\] (see \cite[5.1]{Lu88}, \cite[15.13]{CabEng}). Note that here $\zeta\in\Irr (\Cent_{\wbG^*}(s)^\Fdual)$, being unipotent, is identified with a character of $\Cent_{\bG^*} ^\circ (\pi ( s))^{F^*} = \pi (\Cent_{\wbG^*}(s))^{F^*}=\pi (\Cent_{\wbG^*}(s)^{F^*})$. By the degree formula from \cite[7.1]{Ma06}, the equation
$\chi^\bG_{\pi ( s),\zeta '} (1)_\ell =1$ holds if and only if $|\bG^*{}^{F^*}:\Cent_{\bG^*} (\pi (s))^{F^*}|_\ell =1$ and $\zeta '(1)_\ell =1$. Since $\Cent_{\bG^*} (\pi ( s))^{F^*}/\Cent_{\bG^*} ^\circ (\pi ( s))^{F^*}$ is cyclic, the characters in $\Irr (\Cent_{\bG^*} (\pi ( s))^{F^*}\mid\zeta )$ all have the same degree by Clifford theory. All characters of this set have $\lp$-degree if and only if $\zeta (1)_\ell =1$ and $\ell\nmid |\Cent_{\bG^*} (\pi (s))^{F^*}:\Cent_{\bG^*} (\pi (s))^{\Fdual}_\zeta|$. 
Recall now that \[\zeta =\calI^{\Cent_{\w\bG^*}(s)}_{\Cent_{\w\bG^*}(\bS^*_0 ,s),\la}(\eta )\] for some $\la\in\E ({\Cent_{\w\bG^*}(\bS^*_0,s)^{F^*}},1)$ and $\eta\in\Irr _\lp (\W_{\Cent_{\w\bG^*} (s)}(\Cent_{\w\bG^*}(\bS^*_0,s))_\la ^\Fdual)$. By the equivariance of the generalized $d$-Harish-Chandra induction (see \cite[3.4]{CabSpaeth}), the stabilizer of $\zeta$ in $\Cent_{\bG^*}(\pi (s))^{F^*}$ is the stabilizer of the triple $(\bS_0^*, \la ,\eta )$ modulo $\Cent_{\bG^*}^\circ (\pi (s))^{F^*}$-conjugacy. On the other hand, the condition $\zeta (1)_\ell = 1$ implies $|\bG^*{}^{F^*}:\Cent_{\bG^*} (\pi ( s))^{F^*}|_\ell =1$, by \cite[6.6]{Ma06}. This proves equivalence of the first and third statements in (c).
 
In order to check the equivalence of the second and third statements of (c), we use the simplification mentioned at the start of this proof. So we assume that $\bC=\Cent_{\bG}(\bS_0)$ is a torus $\bT_0$ and that $\la =1$. We also denote the dual groups $\w\bC^*=\w\bT_0^*\twoheadrightarrow \bC^*= \bT_0^*$. Then $\psi^{(\w N)} (s,1 ,\eta )=\Ind^{\w N}_{\w N_\xi}(\w\Lambda (\xi ) (\eta\circ i_{s, 1}))$ where $\xi =\chi^{\w\ovC}_{s,1 }$ and $i_{s,1 }$ is the isomorphism $ \W_{\wbG^F}(\bT_0)_\xi\to \W_{\Cent_{\w\bG^*}(s)}(\w\bT_0^*)^{F^*} $ as in \cite[3.3]{CabSpaeth}. Let us show that $\psi^{(\w N)} (s,1 ,\eta )$ covers an element of $\Irr_\lp (N)$ if and only if the two conditions of the third statement are satisfied.
 
As in the proof of \cite[3.3]{CabSpaeth}, $s\in\w\bT_0^*{}^{F^*}$ (respectively $\pi (s)\in \bT_0^*{}^{F^*}$) identifies by duality with a linear character $\xi$ of $\w\bT_0^F$ (respectively $\xi '\deq \Res^{\w\bT_0^F}_{\bT_0^F}\xi$) and $\w N_\xi /\w\bT_0^F \cong\mathrm{W}_{\Cent_{\w\bG^*}(s)}(\w\bT_0^*)^{F^*}$ (respectively $ N_{\xi '} / \bT_0^F \cong\mathrm{W}_{\Cent_{ \bG^*}(\pi (s))}( \bT_0^*)^{F^*}$) by the duality morphism (see also \cite[7.7]{Ma06}).
 
Let us now denote $\eta_s=\eta\circ i_{s, 1}$. We have \[\Res^{\w N}_N(\psi^{(\w N)} (s,1 ,\eta ))=\Res^{\w N}_N(\Ind^{\w N}_{\w N_\xi}(\w\Lambda (\xi ) \eta_{s}))= 
 \Ind^{ N}_{ N_\xi}(\Res^{\w N_\xi}_{N_\xi}(\w\Lambda (\xi ) )\eta_s) \] by the Mackey formula with $\w N =N \w\bT_0^F=N \w N_\xi$.
 
 Lemma~\ref{IndN'N} below tells us that the components of this restriction all have same degree and that it is prime to $\ell$ if and only if $N/N_{\xi '}$ and $N_{\xi '}/N_{\xi ' , \eta_s}$ are $\ell '$-groups.
 Or as we have seen above in terms of centralizers of $s$ and $\pi (s)$ and remembering that $\Cent_{\w\bG^*}(s) = \Cent_{\w\bG^*}^\circ(s) =\pi^{-1}(\Cent^\circ_{\bG^*}(\pi (s))$, 
 \[ |\mathrm{W}_{ {\bG^*}}(\bT_0^*)^{F^*}:\mathrm{W}_{\Cent_{ \bG^*}(\pi (s))}( \bT_0^*)^{F^*}|_\ell = 1\ \ \text{and} \ \ |\mathrm{W}_{\Cent_{ \bG^*}(\pi (s))}( \bT_0^*)^{F^*}: \mathrm{W}_{\Cent_{ \bG^*}(\pi (s))}( \bT_0^*)^{F^*}_\eta |_\ell =1.\]
 
 The first condition above is equivalent to $|\bG^*{}^{F^*}:\Cent_{\bG^*}(\pi (s))^{F^*}|_\ell = 1$ since the normalizer of $\bT_0^*$ in $\bG^*{}^\Fdual$ contains a Sylow $\ell$-subgroup of $\bG^*{}^\Fdual$ (see \cite[3.4(4)]{BrMa}) and $\NNN_{\bG^*}(\bS_0^*)^\Fdual=\NNN_{\bG^*}(\bT_0^*)^\Fdual$ (since $\bT_0^*=\Cent_{\bG^*}(\bS_0^*)$ and $\bS_0^*$ is the unique Sylow $\Phi_d$-torus of $(\bT_0^*,F^*)$). Since Sylow $\Phi_d$-tori of $(\Cent_{ \bG^*}^\circ (\pi (s)), F^*)$ are conjugate under $\Cent_{ \bG^*}^\circ (\pi (s))^{F^*}$, and using again $\NNN_{\bG^*}(\bS_0^*)^{F^*}=\NNN_{\bG^*}(\bT_0^*)^\Fdual$, the last part of the third statement in (c) is equivalent to the fact that the $\ell$-elements of $\NNN_{\Cent_{ \bG^*}(\pi (s))}( \bT_0^*)^\Fdual$ stabilize $\eta$. This is the second condition above.
\end{proof}

\begin{lem}\label{IndN'N}
Let $T\lhd  N' \leq N$ be finite groups with $T\lhd N$. Let $\la\in\Irrl(T)$ and $\phi\in\Irr ({N_{\la}})$ such that $\phi$ is an extension of $\la$. We assume moreover that $N'\lhd {N_{{\la}}}$ and $N_\la /N'$ is cyclic. Let $\theta \in\Irrl (N'/T)$. Then all components of $\Ind^N_{N'}((\Res^{N_{\la}}_{N'}\phi )\theta )$ have the same degree. This degree is prime to $\ell$ if and only if $|N:N_{{\la}}|_\ell = |N_{{\la}}:(N_{\la})_{ \theta}|_\ell = 1$.
\end{lem}

\begin{proof} Let $\w\theta\in\Irr ((N_{\la})_\theta )$ be an extension of $\theta$ whose existence is ensured by the fact that $N_{\la}/N'$ is cyclic. Then
\begin{align*}\Ind^{N_{\la}}_{N'}((\Res^{N_{\la}}_{N'}\phi )\theta )&=  \phi \,\Ind^{N_{\la}}_{N'}(\theta ) = \phi \,\Ind^{N_{\la}}_{(N_{\la})_\theta}(\Ind^{(N_{\la})_\theta}_{N'}(\theta ) )\\ &=  \phi \, \Ind^{N_{\la}}_{(N_{\la})_\theta}\left (\w\theta\sum_{\mu\in\Irr ({(N_{\la})_\theta}/{N'})}\mu\right ) = \sum_{\mu\in\Irr ({(N_{\la})_\theta}/{N'})} \phi \,\Ind^{N_{\la}}_{(N_{\la})_\theta} (\mu\w\theta).\end{align*} Thanks to Clifford's theorem applied to $N'\lhd N_{\la}$ and $\theta$, each $\theta_\mu \deq \Ind^{N_{\la}}_{(N_{\la})_\theta} (\mu\w\theta)$ belongs to $\Irr (N_{\la})$ and has $T$ in its kernel. So applying again Clifford's theorem, this time to $T\lhd N$ and $\la$, each $\Ind^N_{N_{\la}}(\phi \theta_\la )$ belongs to $\Irr (N)$. Writing now \[ \Ind^N_{N'}((\Res^{N_{\la}}_{N'}\phi )\theta )= \sum_{\mu\in\Irr ({(N_{\mu})_\theta}/{N'})} \Ind^N_{N_{\la}}(\phi \theta_\mu )\ ,\] this gives our claim since $\Ind^N_{N_{\la}}(\phi \theta_\mu )$ has degree $|N:N_{\la} |  \la(1)\theta_\mu (1)=|N:(N_{\la})_\theta |   \la(1)\theta (1)$. \end{proof} 

The following is related with Isaacs-Navarro's refinement of McKay's conjecture stated in \cite[Conjecture A]{IsaacsNavarro}.

 \begin{rem}\label{wOmIN} Let $(s,\la ,\eta )\in \w\calM_N$ satisfying the conditions of Proposition~\ref{MalBij}. Then the above proof allows us to relate easily the elements of $\Irr (\bG^F\mid \psi^{(\w G)}(\bS_0^*,s,\la ,\eta ))\subseteq\Irr_\lp (\bG^F)$ and of $\Irr (N\mid \psi^{(\w N)}( s,\la ,\eta ))\subseteq\Irr_\lp (N)$ with Malle's parametrization in \cite[7.5]{Ma06}. In particular \cite[7.8.(d)]{Ma06} implies that for every $(s,\la,\eta)\in \w\calM_N$ the characters of $\bG^F$ and $N$ covered by $ \psi^{(\w G)}(\bS_0^*,s,\la ,\eta )$ and $ \psi^{(\w N)}(s,\la ,\eta )$, respectively, have degrees $r$ and $r'$ such that $r\equiv \pm r'\mod \ell$.
 \end{rem}

We can now finish the proof of the main result in this part.

\renewcommand{\proofname}{Proof of Theorem~\ref{thm6_1}}
\begin{proof}
The map 
 \[\w\Omega\colon \Irr (\w\bG^F | \Irr_\lp ({\bG^F}))\longrightarrow \Irr (\w N |\Irr_\lp (N))\] 
 with 
 \[\w\Omega \left (\psi^{(\w G)}(\bS_0^*,s,\la ,\eta )\right )=\psi^{(\Nu)}(s,\la ,\eta ) \text{ for each }(s,\la ,\eta )\in\w\calM_N \]
 is a well-defined bijection thanks to Proposition \ref{MalBij}. This is a variant of the main result of \cite[Sect. 4]{CabSpaeth}. Note that the Jordan decomposition of characters we use can be taken $(\wbG\rtimes D)$-equivariant, since $\wbG$ has connected center, see \cite[3.1]{CabSpaeth}. Moreover the extension map $\w \La$ with respect to $\w C\lhd \w N$ from Corollary~\ref{cor5_17} is $(\wGF\rtimes D)_{\bS_0}$-equivariant. Furthermore the used maps from $d$-Harish-Chandra theory satisfy similar equivariance properties. According to the proof of \cite[4.5]{CabSpaeth}, $\w\Omega$ is then $(\w\bG^F\rtimes D)_{\bS_0} $-equivariant.

By the considerations in \cite[4.3]{CabSpaeth} we have
\[ \Omegau(\Irr (\w\bG^F | \Irr_\lp ({\bG^F}))\cap\Irr(\w\bG^F \mid \nu))\subseteq \Irr(\Nu\mid \nu) \text{ for every } \nu \in \Irr(\Z(\w\bG^F)).\]

In the next step we verify
\[\Omegau(\chi \delta)= \Omegau(\chi) \Res^{\wbG^F}_{\Nu}(\delta) \text{ for every } \chi, \delta\in\Irr (\wbG^F\mid\Irrl (\bG^F)) \text{ with } \GF\leq \ker(\delta). \]
When $\delta \in \Irr(\w\bG^F)$ with $\GF\leq \ker(\delta)$, there exists some $z\in\Zent ((\w\bG^*)^{F^*})$ such that $\delta =\chi_{z,1}^{\w\bG^F}$ (see \cite[8.20, 15.8]{CabEng}). Let $(s,\la ,\eta )\in\w\calM_N$ and $\zeta\deq \calI^{\Cent_{\w\bG^*}(s)}_{\Cent_{\w\bG^*}(\bS^* ,s),\la}(\eta )$. 
Recall that for $\wbGF=\GL_n(\epsilon q)$ the Jordan decomposition is uniquely defined by scalar products with Deligne-Lusztig characters $\R_{\wbT'}^\wbG(\theta)$, see \cite[15.8]{CabEng}, where $\wbT'$ is a maximal torus of $\wbG$ and $\theta\in\Irr(\wbT'^F)$. Moreover \cite[11.5(b)]{Cedric} implies $\R_{\wbT'}^{\wbG}(\theta)\delta= \R_{\wbT'}^{\wbG}(\theta \, \Res_{\wbT'^F}^{\wbG^F}(\delta))$. Altogether this proves
\[\left( \psi^{(\w G)}(\bS_0^*,s,\la ,\eta )\right) \,\,\delta = \chi_{s,\zeta}^{\w\bG}\,\,\delta =\chi_{sz,\zeta}^{\w\bG}= \psi^{(\w G)}(\bS_0^*,sz,\la ,\eta ).\] 

On the other hand, for $\xi ' \deq \chi_{s,\la}^{\w\ovC}$ and $\eta_s=\eta\circ i_{s,\la}$ the character $\psi^{(\w N)}(s,\la ,\eta ) \,\, \Res^{\wbG^F}_{\w N}(\delta) $ satisfies: 
\begin{align*} \psi^{(\w N)}(s,\la ,\eta )
\,\, \Res^{\wbG^F}_{\w N}(\delta) 
=\Ind^{\w N}_{\w N_{\xi '}} \left (\w\La (\xi ')\, \eta_s\right ) \,\,& \Res^{\wbGF}_{\w N}(\delta) = \Ind^{\w N}_{\w N_{\xi '}}\left (\w\La (\xi ')\,\eta_s\, \Res^{\wbG^F}_{\w N_{\xi '}}(\delta) \right ) 
\end{align*}
The extension map $\w\Lambda$ from Corollary \ref{cor5_17} satisfies 
$\w\La (\xi ')\,\Res^{\wbG^F}_{\w N_{\xi '}}(\delta)= 
\w\La \left (\xi' \, \Res^{\wbG^F}_{\w C}(\delta)\right )$ and since $\chi_{s,\la}^{\w\ovC} \, \Res^{\wbG^F}_{\w C}(\delta)=\w\La ( \chi_{zs,\la}^{\w\ovC})$ we see
\begin{align*} \psi^{(\w N)}(s,\la ,\eta )
\,\, \Res^{\wbG^F}_{\w N}(\delta) 
&= \Ind^{\w N}_{\w N_{\xi '}}\left ( \w\La \left (\chi_{s,\la}^{\w\ovC} \, \Res^{\wbG^F}_{\w C}(\delta) \right ) \,\, \eta_s \right ) \\ 
&= \Ind^{\w N}_{\w N_{\xi '}}\left ( \w\La ( \chi_{zs,\la}^{\w\ovC}) \,\, \eta_s\right )=\psi^{(\w N)}(sz,\la ,\eta ). \text{} 
\end{align*} 
So indeed $\w\Omega (\chi \delta) =\w\Omega (\chi ) \Res^{\wbG^F}_\Nu (\delta)$ for every $\chi\in\Irr (\w\bG^F | \Irr_\lp ({\bG^F}))$ and $\delta\in\Irr(\wGF)$ with $\GF\leq \ker(\delta)$. Our proof is complete.
\end{proof}
\renewcommand{\proofname}{Proof}

\section{Consequences and main results}\label{sec7}

\noindent The considerations of the preceding three sections essentially ensure that the assumptions made in Theorem \ref{thm2_2} are true for most simple groups of type $\tA_{n-1}$ and primes different from the defining characteristic. We nevertheless have to deal with two types of mildly exceptional behaviour described below. In the end we consider the Isaacs-Navarro refinement of the McKay conjecture.

\medskip\subsection{Proof of Theorem~\ref{thm_PSL_ist_gut}}\label{subsec7.1}
\hfill\break

\noindent
Recall that $\bG =\SL_n(\o\FF_p )\leq \wbG =\GL_n(\o\FF_p ) $ with $n\geq 2$ and $F\in \{F_q,\gamma F_q\}$ giving rise to $\bG^F=\SL_n(\ep q)\leq \wbG^F =\GL_n(\ep q )$ in the notations of Sect.~\ref{not_G_wG}. Recall $D=\spann<F_p,\gamma>\leq\Aut (\wbG^F)$.

\begin{lem}\label{lem3_4}
Assume that $S\deq \GF/\Z(\GF )=\PSL_n(\epsilon q)$ is a simple group with $n\geq 3$, and that $\GF $ is its universal covering group. Let $\ell\not= p$ be a prime divisor of $|S|$ such that $\epsilon q\not\equiv 2,5\, \mod \,9$ if $\ell =n=3$. Let $d$ be defined as in Sect.~\ref{not_d} and let ${\bS_0}$ be a Sylow $\Phi_d$-torus of $(\bG,F)$. 

Then the assumptions of Theorem~\ref{thm2_2i} are satisfied for $ G=\GF $, $\w G=\wGF $, $D$, $N\deq \norm \bG {\bS_0}^F$ and some Sylow $\ell$-subgroup $Q\leq\GF$ such that $\NNN_\wbG (\bS_0)^F =  \NNN_{\wGF}(Q)N$.\end{lem}

\begin{proof} The conditions related to $\Aut(\GF )$ are recalled in Sect.~\ref{not_G_wG}. According to \cite[5.14 and 5.19]{Ma06}, $\norm \bG {\bS_0}^F$ contains $\norm \GF Q$ for some Sylow $\ell$-subgroup $Q$ of $\bG^F$, if $(\ell ,n)\neq (3,3)$ or $\epsilon q\not\equiv 2,5\, \mod \,9$. From the properties of Sylow $\Phi_d$-tori we see that $\NNN_\bG({\bS_0})^F$ is then an $\Aut(\GF)_Q$-stable subgroup of $\GF$, see for example \cite[Sect.~2.5]{CabSpaeth}. By $\bG^F$-conjugation of Sylow $\ell$-subgroups and Sylow $\Phi_d$-tori we get $\wGF/\GF \cong \NNN_{\wGF}(Q)/\NNN_{\GF}(Q)\cong \NNN_\wbG({\bS_0})^F/\NNN_\bG({\bS_0})^F$ by the natural maps. Since $\NNN_{\wGF}(Q)\leq\NNN_\wbG ({\bS_0})^F$, this implies our last equality. 

Every $\chi\in\Irr(\GF)$ extends to its stabilizer in $\wGF$ since $\wGF/\GF$ is cyclic. The group $\w N\deq  \NNN_{\w G}(Q)N$ coincides with $\NNN_\wbG({\bS_0})^F$, so $\w N/N$ is also cyclic, and again every $\psi\in\Irr(N)$ extends to its stabilizer in $\w N$. 
\end{proof}

\begin{prop}\label{prop7_1} Let $n\geq 2$, $q$ a prime power, and $\ep =\pm 1$.
Assume $S = \PSL_n(\epsilon q)$ is simple non-abelian and such that the corresponding $G=\SL_n(\epsilon q)$ is the covering group of $S$. Then $S$ satisfies the inductive McKay condition from \cite[\S 10]{IsaMaNa} for any prime $\ell$.\end{prop}
\begin{proof}
Note first that according to \cite[1.1]{Spaeth5} we can assume that $\ell\nmid q$. We may also assume $n\geq 3$, thanks to \cite[\S 15]{IsaMaNa}. Note that if moreover $\ell = n=3$ and $\epsilon q\equiv 2,5\, \mod \,9$, then our claim is given by Theorem 3.2 of \cite{Malle_exceptions2}. So we may assume the hypotheses of Lemma \ref{lem3_4}. 

According to Lemma \ref{lem3_4}, the assumptions from Theorem \ref{thm2_2i} are satisfied for $N$, $\w G$ and $D$ defined as above. 
On the other hand Theorem~\ref{GloStaA} ensures hypothesis (ii) of Theorem~\ref{thm2_2}. Theorem~\ref{thm_IrrN_autom} and Theorem~\ref{thm6_1} then verify the remaining hypotheses (iii) and (iv) of that statement since Lemma \ref{lem3_4} above also ensures that $\w N$, which is defined as $N\NNN_{\w G}(Q)$ in Theorem~\ref{thm2_2}, coincides with $\NNN_\wbG (\bS_0)^F$. Consequently, Theorem \ref{thm2_2} applies and the inductive McKay condition holds for $S$.
\end{proof}

In view of the above proposition, we now have to consider simple groups $S=\PSL_n(\ep q)$ ($n\geq 3$) whose covering group $\wh S$ is different from the corresponding $G=\SL_n(\ep q)$. According to the standard terminology, the Schur multiplier decomposes as $\Z(\wh S) =\Z(\wh S)_c\times \Z(\wh S)_e$ where $\Z(\wh S)_c =\Z(G)$ is the so-called canonical Schur multiplier and $\Z(\wh S)_e$ is the exceptional Schur multiplier of $S$. The cases for an $S$ of the above type with $\Z(\wh S)_e\not= 1$ are as follows (see \cite[6.1]{GLS3}, noting that the case of $\PSL_3(2)\cong \PSL_2(7)$ is excluded in our setting). 

\begin{table}[!h] \centering
\begin{tabular}{|c|c|c|c|c|c|}
\hline $S$ & $\PSL_4(2)$&$\PSL_3(4)$&$\PSU_4(2)$&$\PSU_6(2)$&$ \PSU_4(3)$  {} \\ 
\hline $\Z(\wh S)_c$  &$ 1 $&$ \Cy_3 $&$ 1 $&$ \Cy_3 $&$\Cy_4$ \\ 
\hline $\Z(\wh S)_e$  &$\Cy_2  $&$ \Cy_4\times \Cy_4 $&$ \Cy_2 $&$ \Cy_2\times\Cy_2 $&$ \Cy_3\times \Cy_3 $\\ \hline
\end{tabular} 
\medskip
\caption{Exceptional Schur multipliers\label{Table1}}
\end{table}

We sketch below some adaptations of the inductive McKay condition allowing to deal with this kind of exceptions (in any type) in order to make use of Malle's results in \cite{ManonLie}. We also deal with the so-called IN refinement studied later.

\begin{definition}\label{iMK4SZ}
Let $S$ be a simple non-abelian group, and $\ell$ a prime with $\ell\mid |S|$. Let $\wh S$ be the universal covering group of $S$. Let $Q$ be a Sylow $\ell$-subgroup of $\wh S$, let $Z = \Z(\wh S)/U$ be a cyclic $\ell'$-quotient of $\Z(\wh S)$. We say that the {\it inductive McKay condition holds for $(S,Z)$ and $\ell$}, if moreover:
\enumroman\begin{enumerate}
\item There exists an $\Aut(\wh S)_Q$-stable subgroup $N$ with $\norm {\wh S} Q\leq N \lneq \wh S$.
\item For every $\nu\in\Irr(\Z(\wh S))$ with $U=\ker(\nu)$ there exists an $\Aut(\wh S)_{Q,\nu}$-equivariant bijection 
\[\Omega_\nu: \Irrl(\wh S\mid \nu)\longrightarrow \Irrl(N\mid \nu)\]
(where $\Irrl(\wh S\mid \nu):= \Irr(\wh S\mid \nu)\cap \Irrl(\wh S)$ and so on).
\item For every $\nu\in\Irr(\Z(\wh S))$ with $U=\ker(\nu)$ and $\chi\in\Irrl(\wh S\mid \nu)$ there exists some group $A$ such that $\o S\deq  \wh S/U\lhd A$, $\Cent_A(\o S)=\Z(A)$ and $A/\Cent_A(\o S)=\Aut(\o S)_\chi$, $\chi$ as character of $\o S$ extends to some $\w\chi\in\Irr(A)$ and $\Omega_\nu(\chi)$ seen as character of $\o N\deq N/U$ extends to some $\w\chi'\in \Irr(\NNN_{A}(\o N))$. Furthermore $\Irr(\Z(A)\mid \w\chi)=\Irr(\Z(A)\mid \w\chi')$.
\end{enumerate}\enumalph
If moreover for every $\nu\in\Irr(\Z(\wh S))$ with $U\deq \ker(\nu)$ the bijection $\Omega_\nu: \Irrl(\wh S\mid \nu)\longrightarrow \Irrl(N\mid \nu)$ satisfies \[ \chi(1)_{\lp}\equiv \pm \Omega_\nu(\chi)(1)_{\ell'}\mod \ell \text{ for every } \chi\in\Irrl(\wh S \mid \nu)
,\]we say that the {\em inductive IN-condition condition holds for $(S,Z)$ and $\ell$}. 
\end{definition}
Those refined versions are related to the inductive McKay condition of \cite[\S 10]{IsaMaNa} by the following. 
\begin{lem}\label{lem7_4}
Let $S$ be a simple non-abelian group, and $\ell$ a prime with $\ell\mid |S|$. Let $\wh S$ be the universal covering group of $S$. 
\begin{enumerate}
\item The inductive McKay condition holds for $S$ and $\ell$, if it holds for $(S,Z)$ and $\ell$ in the sense of Definition~\ref{iMK4SZ} for every cyclic $\ell'$-quotient $Z$ of $\Z(\wh S)$. 
\item The inductive IN-condition from Definition 3.1 of \cite{Spaeth_AM_red} holds for $S$ and $\ell$, if  it holds for $(S,Z)$ and $\ell$ in the sense of Definition~\ref{iMK4SZ} for every cyclic $\ell'$-quotient $Z$ of $\Z(\wh S)$. 
\item Let $Z$ be a cyclic $\ell'$-quotient of $\Z(\wh S)$. If $(S,Z)$ is very good for $\ell$ in the sense of \cite[p. 456]{ManonLie}, then the inductive McKay condition holds for $(S,Z)$ and $\ell$. 
\end{enumerate}
\end{lem}
 \begin{proof}
The first part follows from a combination of the original definition of the inductive McKay condition and the reformulation of the cohomological condition in \cite[2.9]{Spaeth5}. Part (b) is an easy consequence of that. 
 The results in Section 2 of \cite{Spaeth5} imply that the inductive McKay condition holds for $(S,Z)$ and $\ell$ whenever $(S,Z)$ is very good. 
 \end{proof}
 
We can now complete the proof of Theorem \ref{thm_PSL_ist_gut}.

 Thanks to Proposition~\ref{prop7_1}, we may assume that the exceptional part of the Schur multiplier of $S$ is non-trivial and $S =\PSL_n(\ep q)$ is among the groups of Table \ref{Table1}. Let $\wh S$ be its universal covering group, and $G\deq \SL_n(\epsilon q)$. Note that $\Out (\wh S)\cong \Out (G)\cong\Out (S)$ is still induced by $\GL_n(\ep q)\rtimes D$. Let $Z=\Z(\wh S)/U$ be a cyclic $\ell'$-quotient of $\Z(\wh S)$. If $\Z(\wh S)_e\leq U$ then $Z$ is a quotient of $\Z(G)$, 
and the inductive McKay condition holds for $(S,Z)$ and $\ell$ according to the proofs of Proposition \ref{prop7_1} and of \cite[2.11]{Spaeth5}.
If $Z$ is not a quotient of $\Z(G)$ and $S\in \{\PSL_4(2), \PSL_3(4), \PSU_6(2), \PSU_4(3)\}$, then according to \cite[4.1]{Ma06}, the pair $(S,Z)$ is very good in the sense of \cite[p. 456]{ManonLie}.

There remains the case of $S=\SU_4(2)=\PSU_4(2)\cong \mathrm{PSp}_4(3)$. The primes $\ell =2$ and $3$ can be seen as defining primes so our statement is given by Theorem 1.1 of \cite{Spaeth5} (note also that in the case of $\ell =2$, since $\Z(\wh S)_e$ is a $2$-group we can argue as above and the statement follows from Proposition \ref{prop7_1}). For the remaining prime divisors $\ell = 5,7$ of $|\wh S|$, the corresponding Sylow subgroup of $\hat S$ is cyclic while $\Out (S)\cong\Cy_2$. So we may argue as in \cite[\S 4]{ManonLie}, using \cite[2.1]{ManonLie} on automorphisms of prime order and the fact that any finite group with cyclic Sylow subgroup satisfies McKay conjecture for the corresponding prime. \qed

\medskip\subsection{The Isaacs-Navarro refinement}\label{subsec7.2}
\hfill\break

\noindent In \cite{IsaacsNavarro} the authors introduced a refinement of the McKay conjecture considering the $\ell'$-parts of the degrees. In Theorem A of \cite{Spaeth_AM_red} it has been proven that this refined conjecture can also be reduced to a condition on simple groups. The inductive IN-condition for a simple non-abelian group and a prime $\ell$ requires that the inductive McKay condition holds and the bijection $\Omega: \Irrl(G)\longrightarrow \Irrl(N)$ has an additional property. Here we verify that the IN-condition holds for the simple groups considered in Theorem \ref{thm_PSL_ist_gut}. Our considerations in the preceding sections allow an analogous conclusion on the inductive IN-condition thanks to the following refinement of Theorem \ref{thm2_2}. 
\begin{lem}\label{lem2_4}
Assume that in the situation of Theorem \ref{thm2_2} the bijection $\w\Omega$ satisfies for every $\chi\in\calG\deq \Irr\left (\Gu\mid \Irrl(G)\right )$ the equality
\[ \chi_0(1)_{\lp}\equiv \pm  \psi_0(1)_\lp \,\mod\, \ell \text{ for every }\chi_0\in\Irr(G\mid \chi) \und \psi_0\in\Irr(N\mid \w\Omega(\chi)). \]
Then the inductive IN-condition from Definition 3.1 of \cite{Spaeth_AM_red} holds for $S$ and $\ell$.\end{lem}
\begin{proof}
The bijection $\Omega:\Irrl(G)\longrightarrow\Irrl(N) $ constructed in the proof of Theorem 2.12 of \cite{Spaeth5} satisfies $ \Omega(\Irr(G\mid \chi))=\Irr(N\mid\w\Omega(\chi)) \forevery \chi\in\calG$. Since every $\chi_0\in \Irrl(G)$ extends to its stabilizer in $\w G$ and an analogous property holds for the characters of $\Irrl(N)$, all characters in $\Irr(G\mid \chi)$ and $\Irr(N\mid\w\Omega(\chi))$, respectively have the same degree. Hence $\Omega$ satisfies 
\[ \chi_0(1)_\lp \equiv \pm\Omega(\chi_0) (1)_\lp \,\mod\, \ell \text{ for every }\chi_0\in\Irrl(G). \]
This is the additional requirement that has to be satisfied for the inductive IN-condition from \cite[3.1]{Spaeth_AM_red}. 
\end{proof}

Applied in our situation this leads to the following statement. 
\begin{theorem}\label{thm_IN}
The simple groups $\PSL_n(q)$ and $\PSU_n(q)$ satisfy the inductive IN-condition for any prime.  
  \end{theorem}
\begin{proof}
Let $S=\PSL_n(\epsilon q)$ be a simple group. For $\ell\in\{2,3\}$ the inductive McKay condition and the inductive IN-condition are equivalent. According to Theorem 4.4 of \cite{Spaeth_AM_red} the inductive IN-condition holds for $S$ whenever $\ell\mid q$. If $S$ is one of the groups considered in Proposition \ref{prop7_1} with $5\leq \ell\nmid q$, the bijection $\w \Omega$ from Theorem \ref{thm6_1} is used in order to apply Theorem \ref{thm2_2}. According to Remark \ref{wOmIN}, the bijection $\w \Omega$ satisfies the assumption from Lemma \ref{lem2_4}. Hence in this case the inductive IN-condition holds here as well. 

We are now left with the cases where $S$ has a non-trivial exceptional Schur multiplier and $\ell\geq 5$ divides $|S|$. According to Proposition 4.1 of \cite{Spaeth_AM_red} the proofs from \cite{IsaMaNa} show that the inductive IN-condition holds whenever $n=2$. If $n\geq 3$, $\ell\geq 5$ and $\Z(\wh S)_e\neq 1$, every Sylow $\ell$-subgroups of $ S$ is cyclic (see Table~\ref{Table1}) and hence the so-called inductive AM-condition together with the IN-refinements hold for the $\ell$-blocks of the universal covering group of $S$ according to Theorem B of \cite{KoshitaniSpaeth}. This implies the inductive IN-condition. 
\end{proof}

\section{An equivariant Jordan decomposition of characters}\label{sec8}
\noindent
In this section, we show that Bonnaf\'e's construction of a Jordan decomposition for characters of $\bG^F=\SL_n(q)$ or $\SU_n(q)$ in \cite[\S 5.3]{Bo00}, can be made 
equivariant with respect to a natural action of outer automorphisms defined below (see Theorem~\ref{EquJor}). For this we use the $D$-invariant representations of $\bG^F$ from Theorem~\ref{GammaSta}.

Recall that  $\bG =\SL_n(\o \FF_p )\leq\w\bG =\GL_n(\oFp )$, $\gamma\colon\wbG\to\wbG$ is as in Sect.~\ref{not_G_wG}, $F_p\colon\wbG\to\wbG$ raises matrix entries to the $p$-th power, and $F\in\{ F^m_p,\gamma F^m_p\}$. Let ${\bG^*} \deq  {\PGL}_n(\oFp )$, ${\wbG^*} \deq  {\GL}_n(\oFp )$ and $F^*$ be as in Sect.~\ref{sec_Bij_Gu} above. Note that  for $\si\in D=\spann<\gamma ,F_p>\leq\Aut (\wbG^F)$ and seeing $\bG^*{}^{F^*}$ as quotient of $\wbG^*{}^{F^*}$, we obtain via $\wbG^*{}^{F^*}= \wbG^F$ an automorphism of $\wbG^*{}^{F^*}$ and $\bG^*{}^{F^*}$, that we denote by $\si^*$. 

Let us recall some notation from \cite{Cedric} applying to any connected reductive group with Steinberg endomorphism $(\bG ,F)$ and a dual group $(\bG^*,{F^*})$. For $s\in({\bG^*})^{F^*}_{\sss}$, one denotes by $\E (\bG^F,[s])\subseteq \Irr (\bG^F)$ the rational series  as defined in \cite[\S 11]{Cedric}, see also \cite[14.41]{DigneMichel}. Furthermore \[{\E}({\Cent_{\bG^*} (s)^{F^*}},1)\deq \Irr (\Cent_{\bG^*} (s)^{F^*}\mid {\E}({\Cent^\circ_{\bG^*} (s)^{F^*}},1)),\] i.e. $\E({\Cent_{\bG^*} (s)^{F^*}},1)$ is the set of characters of $\Cent_{\bG^*} (s)^{F^*}$ lying above a unipotent character of $\Cent^\circ_{\bG^*} (s)^{F^*}$, see \cite[\S 27]{Cedric} and \cite[\S 7]{Bon}. Broadly speaking, a Jordan decomposition of characters is a bijection 
\[ \E (\bG^F,[s])\longrightarrow  {\E}({\Cent_{\bG^*} (s)^{F^*}},1) .\]

 Let $\calZ  = \Zent (\bG )/\Zent ^\circ(\bG)$ and $ \calZ_F\deq \calZ /(F-1)\calZ = \Zent (\bG)/(F-1)\Zent (\bG)$. According to \cite[6.3]{Cedric}, one defines a group morphism $ \calZ_F\hookrightarrow \Out (\GF )$ with $z\mapsto\tau_z^\bG$, where $\tau^{\bG}_z$ denotes the outer automorphism associated with the conjugation by any $t\in\bG$ with $t^{-1}F(t)$ belonging to the class $z\in \Zent (\bG)/(F-1)\Zent (\bG)$. The embedding $\bG\leq\w\bG =\GL_n(\oFF_p )$ (or any embedding as in \cite[2.B]{Cedric} with connected $\Zent (\w\bG )$ and $\w\bG =\Zent (\w\bG )\bG$) yields a natural isomorphism $\calZ_F\cong \w\bG^F/(\Zent (\w\bG )^F\bG^F)$ (see \cite[6.B]{Cedric}) through which $\tau_z^\bG$ corresponds to the conjugation by any element of the corresponding class in $\w\bG^F/(\Zent (\w\bG )^F\bG^F)$.
 
  If $s\in\bG^*_{\sss}{}^{F^*}$, one denotes by  $A_{\bG^*}(s)$ the abelian group $\Cent_{\bG^*} (s)/\Cent^\circ_{\bG^*} (s)$, and $\hat\omega_{\bG ,s}^0\colon \calZ_F \twoheadrightarrow \Irr (A_{\bG^*}(s)^{F^*})$ (the latter seen as a multiplicative group) denotes the group morphism defined by duality, see \cite[8.4]{Cedric}.

\begin{definition} Let $\widehat{\Jor}(\bG^F)$ be the set of pairs $(s,\zeta )$ where $s\in (\bG^*)^{F^*}_{\rm ss}$ and $\zeta\in{\E}({\Cent_{\bG^*} (s)^{F^*}},1)$. Clearly $\bG^*{}^{F^*}$ acts on $\widehat\Jor (\bG^F)$. We denote by $\Jor (\bG^F)$ the set of equivalence classes thus defined, i.e.
  \[{\Jor} (\GF)\deq \Bigl(\coprod_{ s\in({\bG^*})^{F^*}_{\rm ss}}{\E}({\Cent_{\bG^*} (s)^{F^*}},1)\Bigl)/{{\bG^*}^{F^*}\text{-conj}}.\]\end{definition}

 \begin{theorem}\label{EquJor}  
 \begin{enumerate}
\item Let $\sigma \in D$ act on $\Jor(\GF)$ via $\sigma^*$ and 
let $z\in\calZ_F$ act by $(s,\zeta )\mapsto (s, \hat\omega^0_{\bG ,s}(z)\zeta )$ for every  $\zeta\in    {\E}({\Cent_{\bG^*} (s)^{F^*}},1)$. 
Thereby $\calZ_F \rtimes D$ acts on $\Jor(\GF)$.
\item With respect to this action there exists a $\calZ_F\rtimes D$-equivariant bijection 
\[\Irr (\GF )\longrightarrow {\Jor} (\GF )\] inducing bijections $\E(\GF, [s])\longrightarrow  {\E}({\Cent_{\bG^*} (s)^{F^*}},1)$ for each $s\in({\bG^*})^{F^*}_{\sss}$. 
 \end{enumerate}
\end{theorem}

 
  \begin{rem}\label{EqJoRem}
  \enumroman
  \begin{enumerate}
  \item The conjecture $(\mathfrak G)$ from \cite[\S~14]{Cedric} relates Deligne-Lusztig restriction  with (ordinary) Gelfand-Graev characters. Assuming it, Bonnaf\'e constructed a Jordan decomposition of characters for $G=\SL_n(q)$ (adapted to the case of $\SU_n(q)$ by the first author) that is compatible with Lusztig functors, see \cite[\S 27]{Cedric} and \cite[4.9]{C13}. Here we construct an equivariant Jordan decomposition using the ideas from \cite[\S 5.3]{Bo00}. This can be done unconditionally but it is not clear if the map obtained coincides with the one of \cite[\S 27]{Cedric} or \cite[4.9]{C13}. 
	
   \item  Note that the statement of Theorem~\ref{EquJor} can be conjectured to hold for any finite reductive group $\HH^F$. If $\HH$ has connected center it is a consequence of  \cite[3.1]{CabSpaeth}.
  \end{enumerate}
   \end{rem} 
  
  The proof of Theorem~\ref{EquJor}(a) is an easy consequence of the following: 
  \begin{lem}\label{omesig} If $s\in{\bG^*}{}^{F^*}_{\rm ss}$ and $\si\in D$, then
  $\hat\omega_{\bG ,\si^*(s)}^0\circ \si  =\si^*\circ\hat\omega_{\bG ,s}^0$  on $ \calZ_F$.
  \end{lem}
  
  \begin{proof} It clearly suffices to prove this for $\si = F_p$ and $\gamma_0$. Recall that $\Zent (\wbG ^*)$ is the kernel of the map $\w\bG^*\twoheadrightarrow{\bG^*}$ dual to the embedding $\bG\hookrightarrow\w\bG$. Recall the map $\omega_{s}\colon A_{\bG^*} (s)\to \Zent (\wbG ^*)$ induced by $c\mapsto [\w c,\w s]$ for $c\in\Cent_{\bG^*} (s)$, and $\w c,\w s\in\w\bG^*$ mapping to $c,s$ (see  \cite[\S 8.A]{Cedric}). We have $\si^*\omega_s =\omega_{\si^*(s)}\si^*$ on $A_{\bG^*} (s)$. Denoting by $\bT_0$ and $\bT_0^*$ the diagonal tori, the duality isomorphism $X(\bT_0) \xrightarrow{\sim} Y(\bT_0^*)$ maps $\si$ to $\si^*$ since both act by the same exponent ($p$ or $-1$ according to $\si = F_p$ or $\gamma_0$) on those tori. The isomorphism $\Zent (\wbG ^*)\cap [\w\bG^* ,\w\bG^*]\cong \Irr (\calZ (\bG ) )$ is a restriction of the above (see \cite[4C]{Cedric}), so it also transforms $\si^*$ into $\si$. Composing this with the restriction of $\omega_s$ to $A_{\bG^*} (s)$, one then gets $\omega_{\bG ,s}\colon A_{\bG^*} (s)\hookrightarrow \Irr (\calZ (\bG))$ (see \cite[8.5]{Cedric}) satisfying $\omega_{\si (\bG ),\si^*(s)}\circ\si^* = \si\circ\omega_{\bG ,s}$. 

From the definition of $\hat\omega_{\bG ,s}^0$ (see \cite[8.4]{Cedric}), our claim now comes by restricting further to $F$-fixed points and taking $\CC$-duals.  \end{proof}
  
 Our main task is now to prove Theorem~\ref{EquJor}(b). An important ingredient is the precise knowledge of the Jordan decomposition of $\wGF$.  
  
\medskip\subsection{Jordan decomposition for $\GL_n(q)$ and $\GU_n(q)$}\label{subsec8.1}
\hfill\break

\noindent Recall that, for each $\w s\in\wbG^*_{\sss}{}^{F^*}$ the group $W(\w s)$ is defined by $W(\w s)\deq {\rm W}_{\Cent_{\wbG^*}(\w s)}(\w\bT_s^* )$ where $\w\bT_s^*$ is a maximally split torus of $\Cent_{\wbG^*}(\w s)$. 

Let $\w s\in\wbG^*_{\sss}{}^{F^*}$. Lusztig-Srinivasan constructed a bijection 
\begin{align}\label{bon_bij}
 \Irr (W(\w s))^{F^*} \longrightarrow  \E (\wbG^F ,[\w s])\end{align} (see \cite[5.1.2]{Bo00}).
For $\eta\in \Irr(W(\w s))$ we denote by $\wh\eta\in\Irr (W(\w s)\rtimes\spann<F>)$ Lusztig's canonical extension of $\eta$ defined as in \cite[\S 35]{Cedric}), $\w\bT^*_w$ denotes an $F$-stable torus of type $w$ with regard to $\w\bT_s^*$, and for $t\in\w\bT^*\leq \wbG^*$, $\R^\wbG_{\w\bT^*}(\w t)$ denotes the class function $\R^\wbG_{\w\bT}(\theta )$, where $(\w\bT^* ,\w t)$ is associated with $(\w\bT ,\theta )$ by duality.
With this notation the map from \eqref{bon_bij}
is given by $\eta  \mapsto  R_{\eta}^{\wbG}[\w s]$ where 
\[R_{\eta}^{\wbG}[\w s] =\pm \frac 1{ |W(\w s)|}\sum_{w\in W(\w s)} \wh\eta (wF) \R^\wbG_{\w\bT^*_w}(\w s).\]
Automorphisms of $D$ induce the following action on those characters. 
 \begin{proposition}\label{JorGLGU} For  every $\si\in D$, $\w s\in\wbG^*_{\sss}{}^{F^*}$ and $\eta\in\Irr(W(\w s))$, the character $R_{\eta}^{\wbG}[\w s]$ satisfies  $R_{\eta}^{\wbG}[\w s] = {}^\si (R_{^{\si^*}\!\eta}^{\wbG}[\si ^*(\w s)])$. 
 \end{proposition}
 
 \begin{proof}  By the definition of $R_{\eta}^{\wbG}[\w s]$ given above, we observe 
 \[ |W(\w s)|\, {}^\si (R_{^{\si^*}\!\eta}^{\wbG}[\si ^*(\w s)])=\pm \sum_{w\in W(\si^*(\w s))}\wh{^{\si^*}\!\eta} (wF)\, \, {}^{ \si^{}}\R^\wbG_{\w\bT^*_w}(\si^*(\w s)).\] 
 We have $W(\si ^*(\w s))=\si^*(W(\w s))$ since $\si^*$ preserves the property of being ``maximally split", and $\wh{^{\si^*}\!\eta} = {}^{\si^*}\!(\wh{\eta})$ (see \cite[35.2]{Cedric}). Moreover, if $(\w\bT^*,\theta )$ is associated with $(\w\bT^*_w,\si^*(\w s))$ by duality, then $(\si (\w\bT^*),(\theta)^{\si^{-1}} )$ is associated with $(\si^*{}^{-1}(\w\bT^*_w), \w s)$ (see \cite[2.6]{CabSpaeth}), while according to \cite[3.4]{C13} we note that $\si^*{}^{-1}(\w\bT^*_w)$ is of type $\si^*{}^{-1}(w)\in W(\w s)$ in $\Cent_{\wbG^*}(\w s)$. Using the equivariance of Deligne-Lusztig's functor from \cite[9.2]{DiMi90}, one actually gets $ {}^\si (R_{^{\si^*}\!\eta}^{\wbG}[\si ^*(\w s)]) = \pm R_{ \eta}^{\wbG}[ \w s]$, hence the equality claimed since both terms are irreducible characters.
 \end{proof}
 
\medskip\subsection{Rational series and  characters of Weyl groups}\label{subsec8.2}
\hfill\break

\noindent The proof of Theorem~\ref{EquJor}(b) is in two steps, each relating one of the two sides involved with representations of Weyl groups. For the next definition, see \cite[8A, 23A]{Cedric}.

\begin{definition}\label{WeyJor} Let $\wh {\Wey} (\bG^* ,F^* )$ be the set of pairs $(s,\eta)$, where $s\in\bG^*_{\sss}{}^{F^*}$ and $\eta\in 
\Irr ({{\rm W}}_{\Cent_{\bG^*}(s)}(\bT^*_s )^{w_s{F^*}})$ for some torus $\bT^*_s$ that is $F^*$-stable and contained in an $F^*$-stable Borel subgroup $\bB_s^*$ of $\Cent^\circ_{\bG^*} (s)$, and where $w_s\in {{\rm W}}_{\Cent^\circ_{\bG^*}(s)}(\bT^*_s )$ is defined as in \cite[23.A]{Cedric}, implying that the action of $w_sF^*$ on ${{\rm W}}_{\Cent^\circ_{\bG^*}(s)}(\bT^*_s )$ (a product of symmetric groups) is by permutation of components.
Note that here all possible choices of $(\bT^*_s , \bB_s^* ,w_s)$  for a given $s$ form a $\Cent_{\bG^*} (s)^{F^*}$-orbit.

There is a natural action of $\bG^*{}^{F^*} $ on $\wh {\Wey} (\bG^* ,F^* )$ by conjugation. We denote by ${\Wey} (\bG^* ,F^* )$ the set of $(\bG^*)^{F^*}$-orbits in $\wh \Wey(\bG^*,F^*)$, i.e.
\[{\Wey} (\bG^* ,F^* )\deq  \Bigl( \coprod_{s  ,  \bT^*_s}\Irr \bigl({\rm W}_{\Cent_{\bG^*}(s)}(\bT^*_s )^{w_s{F^*}}\bigr)\Bigr) /{\bG^*}^{F^*}\text{ -conj}.\] \end{definition}

We define an action of $\calZ_F\rtimes D$ on ${\Wey} (\bG^* ,F^*)$ as follows: let $z\in\calZ_F$ act on each set $\Irr({\rm W}_{\Cent_{\bG^*}(s)}(\bT^*_s )^{w_s{F^*}})$ separately by multiplication by $\hat\omega_{\bG ,s}^0(z )$ (note that $A_{\bG^*}(s)^{F^*}=A_{\bG^*}(s)^{w_sF^*}$ by \cite[23.1]{Cedric}). Let $\si\in D$ act on ${\Wey} (\bG^* ,F^*)$ by $\si^*$ the automorphism of $\bG^*{}^{F^*}$ associated with $\si$. This is consistent thanks to Lemma~\ref{omesig}.

 \begin{theorem}\label{JorA}  There is a $\calZ_F\rtimes D$-equivariant bijection \begin{eqnarray*} {\Wey} (\bG^* ,F^* )  \longrightarrow  \Irr (\GF )  \text{ with }(s,\eta )\longmapsto \chi_{s,\eta} .\end{eqnarray*}
  For any $s\in \bG^*_{\sss}{}^{F^*}$ and $\eta\in \Irr ({\rm W}_{\Cent_{\bG^*}(s)}(\bT^*_s )^{w_s{F^*}})$, we have $\chi_{s,\eta}\in\E(\GF,[s])$.
\end{theorem}

Let us recall first the semi-direct product structure of the Weyl groups involved (see \cite[\S 8.A]{Cedric})

\begin{lem}\label{WsmdA}
 If $A_s$ denotes the subgroup of ${\rm W}_{\Cent_{\bG^*}(s)}(\bT^*_s )$ of elements stabilizing $\bB^*_s$ (see definition~\ref{WeyJor}), one has ${\rm W}_{\Cent_{\bG^*}(s)}(\bT^*_s )={\rm W}_{\Cent^\circ_{\bG^*}(s)}(\bT^*_s )\rtimes A_s$, and $w_s$ commutes with $A_s$ (\cite[23.1]{Cedric}). Moreover ${\rm W}_{\Cent_{\bG^*}(s)}(\bT^*_s )^{}={\rm W}_{\Cent^\circ_{\bG^*}(s)}(\bT^*_s )^{}\rtimes A_s$ and ${\rm W}_{\Cent_{\bG^*}(s)}(\bT^*_s )^{w_s{F^*}}={\rm W}_{\Cent^\circ_{\bG^*}(s)}(\bT^*_s )^{w_s{F^*}}\rtimes (A_s)^{F^*}$  are wreath products along the irreducible components of the Coxeter groups ${\rm W}_{\Cent_{\bG^*}(s)}(\bT^*_s )^{}$ and ${\rm W}_{\Cent_{\bG^*}(s)}(\bT^*_s )^{w_s{F^*}}$. 
\end{lem}

We often abbreviate $W(s)\deq {\rm W}_{\Cent_{\bG^*}(s)}(\bT^*_s )\geq W^\circ (s)\deq {\rm W}_{\Cent^\circ_{\bG^*}(s)}(\bT^*_s )$.

Let us recall Bonnaf\'e's construction of a Jordan decomposition for special linear and unitary groups by use of generalized Gelfand-Graev characters (\cite[\S 5.3]{Bo00}). For a given $(s,\bT^*_s, \bB^*_s)$, Bonnaf\'e introduced in \cite[5.3.6]{Bo00}  ${\mathfrak I}(W^\circ (s),A_{\bG^*}(s),F^*)$, the set of all pairs $(\eta^\circ ,\xi )$ where $\eta^\circ\in\Irr (W^\circ (s))^{F^*}=\Irr (W^\circ (s))^{w_sF}$ and $\xi\in\Irr ((A_s)_{\eta^\circ}^{F^*})$. 

Let $s$ and  $\eta^\circ$ be as above. Let $\calC\in\Uni (\wbG )$ be the unipotent class associated with $R_{{\eta^\circ} }^{\wbG}[\w s]\in \Irr (\wbG^F)$ by \cite[5.2.2]{Bo00}, and let $u\in \calC^F$, such that  $\Gamma_u\in \ZZ_{\geq 0}\Irr (\bG^F)$ defined as in Sect.~\ref{sec4} is $D$-fixed, see Theorem~\ref{GammaSta}. As in  \cite[5.3.2]{Bo00} one defines $R_{\eta^\circ}^\bG[s]_1\in\E (\bG^F ,[s])$ to be the only component common to $\Res^{\wbG^F}_{\bG^F}(R_{{\eta^\circ}}^{\wbG}[\w s])$ and $\Gamma_u$.

By \cite[5.3.3]{Bo00}, it is consistent to define $R_{\eta^\circ}^\bG[s]_\xi\in\E (\bG^F ,[s])$ as $^{g_\xi}(R_{\eta^\circ}^\bG[s]_1)$ where $g_\xi\in  \w\bG^F/(\Zent (\w\bG )^F\bG^F)\cong\calZ_F$ corresponds to $z\in\calZ_F$ and $z$ is defined by $\xi = \Res^{A_{\bG^*}(s)^{F^*}}_{A_{\bG^*}(s)^{F^*}_{\eta^\circ}}( \hat\omega_{\bG ,s}^0 (z))\in \Irr(A_{\bG^*}(s)^{F^*}_{\eta^\circ})$. With this definition $D$ acts on those characters as follows. 

\begin{lem}\label{sigmReta}
If $\si\in D$, then $R_{\eta^\circ}^\bG[s]_\xi = {}^\si(R_{^{\si^*}{\eta^\circ}}^\bG[\si^*(s)]_{^{\si^*}\xi})$
\end{lem}

\begin{proof} Since $\Gamma_u$ is $\si$-fixed by Theorem~\ref{GammaSta}, Proposition~\ref{JorGLGU} implies
$R_{\eta^\circ}^\bG[s]_1 = {}^\si(R_{^{\si^*}{\eta^\circ}}^\bG[\si^*(s)]_1)$.

On the other hand, using Lemma~\ref{omesig} and the definition of $g_\xi$, one gets $g_\xi=\si (g_{{}^{\si^*}\!\xi})$. Then our claim follows from the above and the definition of $R_{\eta^\circ}^\bG[s]_\xi$ as $^{g_\xi}R_{\eta^\circ}^\bG[s]_1$.
\end{proof}

 \renewcommand{\proofname}{Proof of~Theorem~\ref{JorA}}
 \begin{proof}
      From Lemma~\ref{sigmReta} and the other properties of Bonnaf\'e's parametrization of $\E (\bG^F ,[s])$ by ${\frakI}(W^\circ (s),A_{\bG^*}(s),F^*)$ recalled above, we now get a $\calZ_F\rtimes D$-equivariant bijection \[ {\frakI}({\bG^*},F^*)\deq \Bigl(\coprod_{s\in\bG^*_{\rm ss}{}^{F^*}}{\frakI}(W^\circ (s),A_{\bG^*}(s),F^*)\Bigr)/\bG^*{}^{F^*}\text{-conj}\ \ \longrightarrow \ \ \Irr (\bG^F).\]
 
 It then suffices to give a $\calZ_F\rtimes D$-equivariant bijection ${\frakI}(\bG^* ,F^*)\longrightarrow {\Wey}(\bG^* ,F^*)$. From their definitions, it suffices to show that ${\frakI}(W^\circ (s),A_{\bG^*}(s),F^*)\xrightarrow{\sim} \Irr (W(s)^{w_sF})$ in a natural way.
 
 This in turn is due to Clifford theory for the semi-direct product $W(s)^{w_sF}=W^\circ (s)^{w_sF}\rtimes A_s^{F^*}$ with \[\Irr (W^\circ (s))^{F^*}=\Irr (W^\circ (s))^{w_sF}\cong \Irr (W^\circ (s)^{w_sF})\] the last equality being a consequence of the wreath product action of $w_sF$ on $W^\circ (s)$ seen in Lemma~\ref{WsmdA}.
\end{proof}

\medskip\subsection{Unipotent characters and wreath products}\label{subsec8.3}
\hfill\break

\noindent In order to prove Theorem \ref{EquJor} we have to compose the bijection from Theorem \ref{JorA} above with a second equivariant bijection between ${\Wey} (\bG^* ,F^* )$  and  ${\Jor} (\bG ,F )$. It is given by the following.
 
 \begin{theorem}\label{UnipCh}  There is a $\calZ_F\rtimes D$-equivariant bijection
 \begin{eqnarray*} {\Wey} (\bG^* ,F^* )  \longrightarrow {\Jor} (\bG ,F )  \text{ with }
(s,\eta )\mapsto R^s_{\eta} \in\E({\Cent_{\bG^*}(s)^{F^*}},1).\end{eqnarray*}

 \end{theorem}
 \renewcommand{\proofname}{Proof}
\begin{proof}
Recall from \cite{B99}, \cite{C13}, the existence of a bijection \[ R^s\colon \Irr \bigl({\rm W}_{\Cent_{\bG^*}(s)}(\bT^*_s )^{w_s{F^*}}\bigr)\longrightarrow {\E}({\Cent_{\bG^*} (s)^{F^*}},1)\ {\text ,\ \ \  } \eta\mapsto R^s_\eta\ \ ,\] related with the wreath product structure of $\Cent_{\bG^*} (s)$ (see \cite[2.5.(ii)]{C13}). In our case, for a choice of ${F^*}$-stable $\bT^*_s\leq\bB^*_s$ and $w_s$ as in Definition~\ref{WeyJor}, we have the wreath product $W(s)\deq {\rm W}_{\Cent_{\bG^*}(s)}(\bT^*_s )^{w_s{F^*}} = W^\circ (s)\rtimes A(s)$  where $ W^\circ (s)= {\rm W}_{\Cent^\circ_{\bG^*}(s)}(\bT^*_s )^{w_s{F^*}}$ and $ A(s)=A_s^{F^*}$ for $A_s$ the stabilizer of $\bB^*_s$ in $W(s)$ (see  Lemma~\ref{WsmdA}). 
The bijection $\Irr (W(s))\to {\E}({\Cent_{\bG^*} (s)},1)$ is defined as follows (see \cite[3.6]{C13}). For a pair $(\eta^\circ  ,\xi )$ with $\eta^\circ \in\Irr ( W^\circ (s))$, $\xi\in \Irr (W(s)_{\eta^\circ }/W^\circ (s))$, the $W(s)$-class of $(\eta^\circ  ,\xi )$ defines a unique $\eta\deq \Ind^{W(s)}_{W(s)_{\eta^\circ }}(\xi.\w{\eta^\circ } )\in\Irr (W(s))$ where $\w{\eta^\circ }$ is the canonical extension of $\eta^\circ $ whose existence is ensured by the wreath product structure (see \cite[1.3]{C13}). The image of $\eta$ by the above map is some $R_\eta^s\in{\E}({\Cent_{\bG^*} (s)^{F^*}},1)$. This is defined by  \[ R^s_\eta =\pm \Ind^{\Cent_{\bG^*} (s)^{F^*}}_{\Cent^\circ_{\bG^*} (s)^{F^*}\rtimes A(s)_{\eta^\circ }}(\xi\w R_{\eta^\circ }^s)\] where $\pm\w R^s_{\eta^\circ }\in \Irr ({\Cent^\circ_{\bG^*} (s)^{F^*}\rtimes A(s)_{\eta^\circ }})$ is defined as explained below and only depends on $\eta^\circ $ (see \cite[7.3]{B99}, \cite[3.3]{C13}). Note that this already implies that our bijection is equivariant for the action of $\calZ_F$. Indeed when the outer automorphism is $\tau_z$ for some $z\in \calZ_F$, this action is by multiplying $\xi$ by the corresponding linear characters $\hat\omega^0_{\bG ,s}(z)$ of  $A(s)$ since the latter is also a representative system of $\Cent_{\bG^*} (s)^{F^*} /\Cent^\circ_{\bG^*} (s)^{F^*}$. 

Let us now look at the action of some $\si\in D$, which we identify with the associated $\si^*\colon {\bG^*}\to{\bG^*}$.  For our claim it is sufficient to prove 
\begin{align}\label{eqR}{} ^\si\w R_{\eta^\circ  ,\bB^*_s ,\bT_s^*,w_s}^s &=  \w R_{^\si\eta^\circ  ,\si (\bB^*_s),\si (\bT^*_s),\si (w_s)}^{\si (s)}. \end{align}
(Note that we have mentioned the choice of the ${F^*}$-stable $\bT^*_s\leq\bB^*_s$ and the element $w_s$, though all choices are $\Cent^\circ_{\bG^*} (s)^{F^*}$-conjugate, hence produce the same character.) The above gives indeed our claim since $(\bB^*_{\si (s)}, \bT^*_{\si (s)},w_{\si (s)})$ is clearly $\Cent^\circ_{\bG^*} (\si (s))^{F^*}$-conjugate to $(\si (\bB^*_s),\si (\bT^*_s),\si (w_s))$. 

In order to prove (\ref{eqR}), we must recall the definition of $\w R_{\eta^\circ  ,\bB_s^*,\bT_s^*,w_s}^s$ (see \cite[3.3]{C13}). For a normal inclusion of finite groups $H'\lhd H$ and $h\in H$, one denotes by $\mathcal{C} (H'h)$ the space of functions $H'h\to \CC$ that are invariant under $H'$-conjugation. For $a\in A(s)_{\eta^\circ }$, $g\in \Cent^\circ_{\bG^*} (s)^{F^*}$  one defines
\[\w R^s_{\eta^\circ }(ga)= {1\over |{\rm W}_{\Cent^\circ_{\bG^*}(s)}(\bT^*_s )^{\spann< a>}|}\sum_{w\in {\rm W}_{\Cent^\circ_{\bG^*}(s)}(\bT^*_s )^{\spann< a>}}(\eta^\circ_a)(w w_s{F^*})\,\, {\rm R}_{(\bT^*_s\spann< a>)_w}^{\Cent^\circ_{\bG^*} (s)\spann< a>}(1)(ga). \] 
In the above, $\eta^\circ _a\in \mathcal{C} ({\rm W}_{\Cent^\circ_{\bG^*}(s)}(\bT^*_s )^{\spann< a>}\, w_s{F^*} )$ is associated with $\eta^\circ \in\Irr ({\rm W}_{\Cent^\circ_{\bG^*}(s)}(\bT^*_s )^{w_s{F^*}})^{\spann< a>}$ by the identifications of spaces of central functions
\begin{align}\label{eqWr} \calC ( W^\circ (s)^{w_s{F^*}} a)\cong \calC ( W^\circ (s)^{\spann<{w_s{F^*} ,a}>} )\cong \calC (W^\circ (s)^{\spann< a>}{w_s{F^*}} ) &= \calC (W^\circ (s)^{\spann< a>}{{F^*}} ) \end{align} due to the wreath product structure (see \cite[\S 1.2]{C13}).
The notation $(\bT^*_s\spann< a>)_w$ denotes any torus of $\Cent^\circ_{\bG^*} (s)\spann< a>$ of type $w$ with regard to $\bT^*_s$ and the Steinberg endomorphism $w_s{F^*}$ (for which $\bT_s^*$ is indeed a diagonal torus, see proof of \cite[4.9]{C13}). The functor ${\rm R}_{(\bT^*_s\spann< a>)_w}^{\Cent^\circ_{\bG^*} (s)\spann< a>}$ (applied here to the trivial representation of $\bigl({(\bT^*_s\spann< a>)_w}\bigr)^{w_s{F^*}}$) is a generalization of Deligne-Lusztig induction to non-connected groups (see \cite[\S 6.3]{B99}).

Now, by canonicity of (\ref{eqWr}) above, we have $(^\si\eta^\circ )_{\si (a)}(\si (w)\si (w_s){F^*} )= (\eta^\circ_a)(w w_s{F^*})$. Equivariance of $\ell$-adic cohomology also implies that $^\si \bigl(  {\rm R}_{(\bT^*_s\spann< a>)_w}^{\Cent^\circ_{\bG^*} (s)\spann< a>}(1) \bigr) = {\rm R}_{\si \bigl( (\bT^*_s\spann< a>)_w\bigr)}^{\Cent^\circ_{\bG^*} (\si (s))\spann<{\si (a)}>}(1)$. On the other hand, $\si \bigl( (\bT^*_s\spann< a>)_w\bigr)$ is indeed of type $\si (w)$ with regard to $\si(\bT^*_s)$ by the same argument as in \cite[3.4]{C13}, replacing there $b$ by $\si$ and noting that $\si$ changes $w_s{F^*}$ into $w_{\si (s)}{F^*}$. In all, we therefore get $^\si \w R^s_{\eta^\circ }(g'a') = \w R^{\si (s)}_{^\si\eta^\circ }(g'a')$ for any $g'\in\Cent^\circ_{\bG^*} (\si (s))^{F^*}$, and $a'\in A(\si (s))_{^\si\eta^\circ }=\si (A(s)_{\eta^\circ })$. This proves (\ref{eqR}).
\end{proof}

\end{document}